\documentclass[reqno]{amsart}

\usepackage{amsmath,amsthm,amssymb,amscd}
\usepackage[arrow,matrix]{xy}
\usepackage[dvips]{graphicx}
\usepackage{enumerate}
\usepackage{hyperref}

\numberwithin{equation}{section}
\newtheorem{thm}[equation]{Theorem}

\newtheorem{cor}[equation]{Corollary}
\newtheorem{lem}[equation]{Lemma}
\theoremstyle{definition}
\newtheorem{dfn}[equation]{Definition}
\newtheorem{exm}[equation]{Example}

\newtheorem{rem}[equation]{Remark}

\def\pd{\partial}

\def\dim{\mathop{\mathrm{dim}}\nolimits}
\def\Lie{\mathop{\mathrm{Lie}}\nolimits}
\def\End{\mathop{\mathrm{End}}\nolimits}

\def\CC{\mathbb{C}}

\def\RR{\mathbb{R}}
\def\ZZ{\mathbb{Z}}

\def\r{\right}
\def\l{\left}

\def\Aut{\mathop{\mathrm{Aut}}\nolimits}

\def\Ad{\mathop{\mathrm{Ad}}\nolimits}
\def\ad{\mathop{\mathrm{ad}}\nolimits}

\def\gr{\mathop{\mathrm{gr}}\nolimits}
\def\Gr{\mathop{\mathrm{gr}}\nolimits}

\def\sl2{\mathfrak{sl}_2}
\def\min{\mathop{\mathrm{min}}\nolimits}
\def\max{\mathop{\mathrm{max}}\nolimits}

\def\half{\frac{1}{2}}
\def\nilp{\mathrm{nilp}}
\def\M{\mathcal M}

\def\half{\frac{1}{2}}

\title{Asymptotics of Degenerations of Mixed Hodge Structures}
\author[T.~Hayama and G.~Pearlstein]{Tatsuki Hayama and Gregory Pearlstein}
\thanks{}
\date{\today}
\address{Mathematical Sciences Center, Tsinghua University, Haidian District,
Beijing 100084, China}
\email{tatsuki@math.tsinghua.edu.cn}
\address{Department of Mathematics, Texas A\& M University, College Station,
TX}
\email{gpearl@math.tamu.edu}

\begin{document}
\begin{abstract}We construct a hermitian metric on the classifying spaces of
graded-polarized mixed Hodge structures and prove analogs of the strong
distance estimate~\cite{CKS} between an admissible period map and the
approximating nilpotent orbit.  We also consider the asymptotic behavior
of the biextension metric introduced by Hain~\cite{H2}, analogs of the
norm estimates of~\cite{KNU} and the asymptotics of the naive limit Hodge
filtration considered in~\cite{KP2}.
\end{abstract}

\maketitle

\section{Introduction}  Let $X$ and $S$ be smooth complex algebraic 
varieties, and $f:X\to S$ be a morphism such that $f$ is smooth, proper and 
connected.  Then by ~\cite{G}, a choice of projective embedding of $X$ determines 
a polarized variation of Hodge structure $\mathcal H\to S$ of weight $k$ via 
the $k$'th cohomology groups of the fibers.  Parallel translation of the data 
of $\mathcal H$ to a fixed reference fiber $H$ using the Gauss--Manin 
connection then determines an associated period map
\begin{equation}
       \varphi:S\to\Gamma\backslash\mathcal D  \label{eq:period-map}
\end{equation}
where $\Gamma$ is monodromy group and $\mathcal D$ is a classifying space
of pure Hodge structures which are polarized by a non-degenerate bilinear
form $Q$ on $H_{\ZZ}$ of parity $(-1)^k$.  

\par In the case where $S$ is a curve with smooth compactification $\bar S$,
the period integrals which define \eqref{eq:period-map} have at worst 
logarithmic singularities at the punctures $p\in\bar S - S$.  More generally, 
let $\bar S$ be a smooth completion of $S$ such that $\bar S - S$ has only 
normal crossing singularities.  Then, the local monodromy of $\mathcal H$ 
about any singular point $p\in\bar S - S$ is quasi-unipotent.  By passage to 
a finite cover, we can then assume that the local monodromy at $p$ is 
unipotent.  The analysis of the singularities of the period map can therefore 
be reduced to the following case:  Let $\Delta^r\subset \bar{S}$ be a polydisk 
with local coordinates $(s_1,\dots,s_r)$ such that $S\cap\Delta^r$ is the locus
of points where $s_1\cdots s_r\neq 0$.  Then, we can construct a
commutative diagram
\begin{equation}
\begin{CD}
            U^r @> F >>            \mathcal D   \\
            @V s_j = e^{2\pi iz_j} VV    @VVV  \\
            \Delta^{*r} @> \varphi >> \Gamma\backslash\mathcal D
\end{CD}                                           \label{eq:basic-diagram}
\end{equation}
where $U^r$ is the product of upper half-planes with coordinates $z_j$,
and
$$
       F(z_1\ldots ,z_j+1,\ldots ,z_n) = T_j.F(z)
$$
where $T_j = e^{N_j}$ is the monodromy of $\mathcal H$ about $s_j = 0$.

\par The group $G_{\mathbb R} = \text{Aut}_{\mathbb R}(Q)$ acts transitively
on the classifying space $\mathcal D$.  The compact dual $\check{\mathcal D}$
is the orbit inside a suitable flag variety of any point in $\mathcal D$ 
under the action of $G_{\mathbb C} = \text{Aut}_{\mathbb C}(Q)$.  The monodromy
transformations $T_j$ belong to $G_{\mathbb R}$, and hence
$$
       e^{N(z)} = e^{\sum_j\, z_j N_j}\in G_{\mathbb C}
$$
Accordingly, $z\mapsto e^{-N(z)}.F(z)$ is a holomorphic map from 
$U^r\to\check{\mathcal D}$ which is invariant under the deck transformations 
$(z_1,\dots,z_r)\mapsto (z_1,\dots,z_j+1,\dots,z_r)$, and hence descends
to a holomorphic map
\begin{equation}
       \psi:\Delta^{*r}\to\check{\mathcal D}     \label{eq:twisted-period-map}
\end{equation}

\begin{thm}\label{thm:schmid-nilp}(Nilpotent Orbit Theorem)~\cite{S} The map
$\psi$ extends to a holomorphic map $\Delta^r\to\check{\mathcal D}$.  Let
$F_{\infty} = \psi(0)$ denote the limiting Hodge flag.  Then,
\begin{itemize}
\item[(a)] Each $N_j$ is horizontal with respect to $F_{\infty}$, i.e.
$N_j(F_{\infty}^p)\subseteq F_{\infty}^{p-1}$ for all $p$;
\item[(b)] There exists a constant $\alpha\geq 0$ such 
$$
    \text{Im}(z_1),\dots,\text{Im}(z_r)>\alpha\implies
    \theta(z) = e^{N(z)}.F_{\infty}\in\mathcal D
$$
\item[(c)] There exists constants $\beta$ and $K$ such that if
$\text{Im}(z_1),\dots,\text{Im}(z_r)>\alpha$ then
$$
      d(F(z),\theta(z))\leq 
       K(\Pi_{j=1}^r\, \text{Im}(z_j))^{\beta}
        \sum_{j=1}^r e^{-2\pi\text{Im}(z_j)}
$$
for any $G_{\mathbb R}$-invariant metric on $\mathcal D$.
\end{itemize}
\end{thm}

\par The basic defect of this distance estimate is that it
fails to establish convergence of the period map and the nilpotent
orbit when the imaginary parts of $z_1,\dots,z_r$ diverge at very
different rates. In~\cite{CKS}, the authors present an argument by
Deligne which shows that there exist constants $\beta_1,\dots,\beta_r$ and 
$K$ such that if $\text{Im}(z_1),\dots,\text{Im}(z_r)>\alpha$ then
$$
      d(F(z),\theta(z))\leq 
       K\sum_{j=1}^r \text{Im}(z_j)^{\beta_j}e^{-2\pi\text{Im}(z_j)}
$$
for any $G_{\mathbb R}$-invariant metric on $\mathcal D$.  Accordingly, the 
period map and nilpotent orbit converge as
$\text{Im}(z_1),\dots,\text{Im}(z_r)\to\infty$.

\par In particular, by~\cite{CKS}, variations of pure Hodge structure
degenerate to variations of mixed Hodge structure along the boundary
strata, and hence one is naturally led to consider the theory of 
nilpotent orbits for degenerations of mixed Hodge structure.  In analogy
with the pure case~\cite{U}, a variation of mixed Hodge structure 
$\mathcal V\to S$ gives rise to a period map
$$
       \varphi:S\to\Gamma\backslash\mathcal M
$$ 
where $\Gamma$ is the monodromy group of the underlying local system of 
$\mathcal V$ acting on a fixed reference fiber $V$ of $\mathcal V$ and
$\mathcal M$ is a classifying space of mixed Hodge structures on $V$
with a given weight filtration $W$, Hodge numbers and graded-polarizations.
In analogy with the pure case, a mixed period map gives rise to a 
commutative diagram
\begin{equation}
\begin{CD}
            U^r @> F >>            \mathcal M   \\
            @V s_j = e^{2\pi iz_j} VV    @VVV  \\
            \Delta^{*r} @> \varphi >> \Gamma\backslash\mathcal M
\end{CD}                                           \label{eq:mixed-diagram}
\end{equation}
where $\mathcal M$ is an open subset of a homogeneous space 
$\check{\mathcal M}$ upon which a complex Lie group $G_{\CC}$ acts transitively.

\par To obtain an analog of Schmid's distance estimate in the mixed case,
one must first endow $\mathcal M$ with a hermitian structure.  In section 2
of this paper, we describe two different hermitian structures on 
$\mathcal M$ which we call the standard metric and the twisted metric.  In 
section 3, we prove the analog of Schmid's distance estimate $(c)$ for period 
maps of admissible variations of mixed Hodge structure and their associated 
nilpotent orbits with respect to the standard metric.  

\par The fact $\mathcal M$  is no longer the homogeneous space of a semisimple 
Lie group in the mixed case however introduces a distortion factor which seems 
to prevent one from obtaining an analogue of Deligne's stronger distance estimate with respect 
to the standard metric.  The twisted metric contains 
additional factors designed to compensate for this distortion at the cost of 
reducing the symmetry of the metric. In section 4, we prove Deligne's strong 
estimate for the twisted metric.  We also prove Deligne's strong estimate
for unipotent variations of mixed Hodge structure arising in the work
of Hain and Morgan on the mixed Hodge theory of the fundamental group
of a smooth complex algebraic variety.  

\par More precisely, we recall that as a consequence of the $SL_2$-orbit
theorem given any mixed Hodge structure $(F,W)$ there exists an associated
mixed Hodge structure 
\begin{equation}
              (\hat F,W) = (e^{-\epsilon(F,W)}.F,W)         \label{eq:canonical-splitting}
\end{equation}
which is split over $\mathbb R$, and $\epsilon$ is given by certain
universal Lie polynomials \cite[Lemma $(6.60)$]{CKS} in the Hodge components 
of Deligne's $\delta$-splitting \cite[Prop. $(2.20)$]{CKS}.  Given a point
$F\in\mathcal M$, we can therefore define 
\begin{equation}
    \tau(F) = 1 + \sum_{p,q<0}\, \|\epsilon^{p,q}\|^{-\frac{2}{p+q}}  
    \label{eq:sl2-twist}
\end{equation}
where $\sum_{p,q<0}\, \epsilon^{p,q}$ is the decomposition 
(see \eqref{eq:bigrading}) of $\epsilon$ into Hodge components with respect to 
$(F,W)$ and $\|*\|$ is the standard metric with respect to $(F,W)$.  The 
twisted metric is then obtained by rescaling the standard norm of an element 
of Hodge type $(p,q)$ by $\tau^{\frac{p+q}{2}}$.  The resulting metric remains 
invariant under $G_{\mathbb R}$,  but has lower symmetry than the standard metric
(cf. Lemma~\eqref{lem:isometries} and \eqref{lem:twisted-sym}).

\begin{rem} The construction of standard metric appears in~\cite{K}.  
In \S 12 of~\cite{KNU}, the authors consider the twisted metric
attached to the function $y_1$ along a period map, but do not appear to
consider the case where $\tau$ is a function on the classifying space
itself (cf. \S 12.9 of~\cite{KNU}).  In \S 4 of~\cite{KNU2} the authors
consider an analog of the twisted metric on the space of $SL_2$-orbits,
see the discussion of norm estimates at the end of this section for
more details.
\end{rem}

\begin{rem} In the case where the classifying space $\mathcal M$ parametrizes 
mixed Hodge structures with exactly two weight graded quotients which are
adjacent, the group $G_{\mathbb R}$ acts transitively on $\mathcal M$ by 
isometries.  However, a simple transcription of the argument to prove
the strong distance estimate given in \cite[\S 1]{CKS} appears to fail
because the group $G_{\mathbb R}$ is no longer semisimple.  Moreover,
the curvature and other geometric properties of $\mathcal M$ differ
from the pure case~\cite{PP}.
\end{rem}

\begin{rem}\label{rem:splitting} For future reference we record the following 
property of Deligne's $\delta$-splitting: If $g\in GL(V_{\mathbb R})$ preserves
$W$ then $(g.F,W)$ is a mixed Hodge structure (not necessarily
graded-polarized) and $\delta(g.F,W) = \text{Ad}(g)\delta(F,W)$.
Since $\epsilon$ is given by universal Lie polynomials in
the Hodge components of $\delta$, the same formula holds for 
$\epsilon$ as well.
\end{rem}

\par The key technical step in proving these results is a relative
compactness result for period maps of degenerations of mixed Hodge
structure.  To state the result, we recall that Schmid's $SL_2$-orbit
theorem attaches to each nilpotent orbit of pure Hodge structure
$e^{zN}.F$ an $sl_2$-pair $(N,H)$ such that $H$ acts via a real
morphism of type $(0,0)$ on an associated limit mixed Hodge structure
of the nilpotent orbit (see \S 3, \cite{CK}).  More generally~\cite{CKS,CK}, 
given a nilpotent orbit of pure Hodge structure 
$$
         e^{z_1 N_1 + \cdots +z_r N_r}.F
$$
the several variable $SL_2$-orbit theorem gives a commuting family of
representations of $sl_2(\mathbb R)$ with semisimple elements
$$
         H_1,\dots,H_r\in\mathfrak g_{\mathbb R}. 
$$

\par Variations of mixed Hodge structure of geometric origin satisfy
a set admissibility conditions~\cite{SZ} which ensure that if $\mathcal V$ is 
an admissible variation of mixed Hodge structure with weight filtration $W$ 
and unipotent local monodromy transformations $T_j = e^{N_j}$ then the 
relative weight filtration of $W$ with respect to $N_j$ exists for each $j$.
Via the diagram~\eqref{eq:mixed-diagram}, an admissible variation of mixed
Hodge structure over $\Delta^{*r}$ determines an admissible nilpotent orbit
\begin{equation}
         (e^{z_1 N_1 + \cdots +z_r N_r}.F,W)               \label{eq:admissible-nilp}
\end{equation} 
in analogy with Schmid's construction.  In particular, the nilpotent
orbit~\eqref{eq:admissible-nilp} is a nilpotent orbit of pure Hodge structure on 
each graded quotient of $W$, and hence determines a corresponding system of 
semisimple elements $H_j$ acting on $Gr^W$.  

\par A choice of isomorphism ({\it grading}) from $Gr^W$ to the reference 
fiber $V$ can be viewed as a choice of direct sum decomposition 
$$
         V = \bigoplus_j\, V_j
$$
such that $W_k = W_{k-1}\oplus V_k$ for each index $k$.  

\par A construction of Deligne presented in~\cite{Sch} associates to an
admissible nilpotent orbit~\eqref{eq:admissible-nilp} a functorial
grading $Gr^W\cong V$.  Let $H_1,\dots,H_r$ be the corresponding
semisimple endomorphisms of $V$, and let $Y_0$ be the endomorphism
of $V$ which acts as multiplication by $j$ on $V_j$.  Let
\begin{equation}
         t(y) = y_1^{-Y_0/2}\Pi_{j=1}^r y_j^{-H_j/2}       \label{eq:t-y}
\end{equation}

\begin{rem} The construction of $t(y)$ using Deligne systems as described
above appears in~\cite{BP1}.  A different construction of $t(y)$ appears in
the introduction to~\cite{KNU} where it is constructed using the work 
of~\cite{CKS} and the limit of a grading of the nilpotent orbit. In 
the pure case~\cite{CK}, the factor $y_1^{-Y_0/2}$ acts trivially on the 
classifying space, and can be omitted.
We denote by $t^{-1}$ the function given by $y\mapsto t(y)^{-1}$.
\end{rem}

\begin{thm}\label{thm:rel-compact} (\S 7~\cite{BP1}) Let 
$\mathcal V\to\Delta^{*r}$ be an admissible variation of mixed Hodge structure 
with unipotent monodromy with associated period map  $F:U^r\to\mathcal M$ as 
in~\eqref{eq:mixed-diagram} and nilpotent orbit~\eqref{eq:admissible-nilp}.  
Let $t(y)$ be the associated family of automorphisms~\eqref{eq:t-y} and 
$$
     I' = \{ (x_1 + i y_1,\dots,x_r + iy_r)\in U^r \mid 
             y_1\geq y_2\geq\cdots\geq y_r\geq 1\}
$$
Then, the image of $I'$ under the map
$$
       z\in U^r\mapsto t^{-1}(y)e^{-\sum_j\, x_j N_j}.F(z)
$$
is a relatively compact subset of $\mathcal M$.
\end{thm}

\begin{rem} In the pure case, this is Theorem $(4.7)$ of~\cite{CK}.  A special 
case of this result appears in \S 12 of~\cite{KNU}.  A proof of this result
is given in~\cite{BP1} where it plays a crucial role in the analysis of the
asymptotic behavior of normal functions.
\end{rem}

\section*{Applications}

\subsection*{Normal Functions} Let $X$ be a smooth complex projective variety 
of dimension $n$.  Then, for any pair of homologically trivial algebraic 
cycles $\alpha$ and $\beta$ on $X$ of dimension $a$ and $b$ on $X$ with 
disjoint support such that $a + b = n-1$, there is an associated
archimedean height
\begin{equation}
         \langle \alpha,\beta\rangle 
          = -\int_{\alpha}\, G_{\beta}  \label{eq:height-period}
\end{equation}
defined by integration of a suitable Green's current $G_{\beta}$ over $\alpha$.
Moreover, as discussed in section 1.1 of \cite{H}, the height can be
viewed as a period of a subquotient of the mixed Hodge structure on 
$H_{2a+1}(X-|\beta|,|\alpha|)\otimes\mathbb Z(-a)$ with weight graded quotients
$$
        Gr^W_0 = \mathbb Z(0),\qquad 
        Gr^W_{-1} = H_{2a+1}(X)\otimes\mathbb Z(-a),\qquad 
        Gr^W_{-2} = \mathbb Z(1)
$$
In particular~\cite{H}, as the triple $(X,\alpha,\beta)$ vary in a flat family
$(X_s,\alpha_s,\beta_s)$ the pairing \eqref{eq:height-period} corresponds to a 
period of a variation of mixed Hodge structure $\mathcal V\to S$ with weight 
graded quotients
\begin{equation}
         Gr^W_0 = \mathbb Z(0),\qquad
         Gr^W_{-1} = \mathcal H,\qquad
         Gr^W_{-2} = \mathbb Z(1)                   \label{eq:biext-data}
\end{equation} 
such that the extension class
$$
        0 \to \mathcal H \to W_0/W_{-2} \to \mathbb Z(0) \to 0 
$$
corresponds to the normal function $\nu_{\alpha}$ attached to the family of 
homologically trivial cycles $\alpha_s$ whereas the extension class
$$
        0 \to \mathbb Z(1) \to W_{-1} \to \mathcal H\to 0
$$
is dual to the normal function $\nu_{\beta}$ attached to the family $\beta_s$.  
Once these extension classes are fixed, there is a natural action of 
$\mathcal O^*$ on the set of possible variations of mixed Hodge structure with 
these extension classes.  Moreover, as explained in section 5, the resulting 
line bundle
$$
         \mathcal B\to S
$$
carries a natural hermitian metric $h$ (see~\cite{H2}).

\par Suppose now that $S$ is a Zariski open subset of a complex manifold 
$\bar S$ such that $\bar S - S$ is a normal crossing divisor along which the
monodromy is unipotent.  Then, it is natural to extend $\mathcal B$ to a 
line bundle $\bar{\mathcal B}\to\bar S$ by declaring the local extending
sections to be admissible variations of mixed Hodge structures with the
the given extension data of the admissible normal function $\nu_{\alpha}$
and $\nu_{\beta}$ (see~\cite{BP2} for details).  

\par In section 5, we will prove the following result:

\begin{thm}\label{thm:loc-int} Let $\mathcal V\to\Delta^{*r}$ be an
admissible variation of biextension type with unipotent monodromy
and 
$$
       |\mathcal V| = e^{-\varphi}
$$
(see \S 5 for the definition of $|\cdot |$).
Then, $\varphi\in L^1_{loc}(\Delta^r)$, and hence defines a singular
hermitian metric on $\bar{\mathcal B}$ (see~\cite{D}).
\end{thm}

\par In particular, as explained in~\cite{BFNP}, a triple
$(X,L,\zeta)$ consisting of a smooth complex projective variety $X$
of dimension $2n$, a very ample line bundle $L\to X$ and a Hodge
class $\zeta\in H^{n,n}(X,\mathbb Z)$ which is primitive with
respect to $L$ determines a normal function
$$
      \nu_{\zeta} : S \to J(\mathcal H)
$$
where $S$ is the complement of the dual variety of $X$ in $|L|$
and $\mathcal H$ is the variation of Hodge structure of weight
$2n-1$ defined by the hyperplane sections of $X$ with respect 
to $S$.  Using $\nu_{\zeta}$ in place of the normal functions
$\nu_{\alpha}$ and $\nu_{\beta}$ considered above, we obtain a 
biextension line bundle $\mathcal B\to S$.  Moreover, as
the second author will show in joint work with Patrick Brosnan,
the failure of the biextension metric $h$ to extend to 
$\bar{\mathcal B}$ in this setting is equivalent to the 
existence of singularities of $\nu_{\zeta}$ of the type considered
by Griffiths and Green in their study of the Hodge conjecture.
\vskip 10pt

\subsection*{Norm Estimates} In \S 12 of~\cite{KNU}, the authors prove
analogs of the norm estimates of~\cite{CKS} for what is effectively
the twisted metric for $\tau(e^{\sum_j\, y_j N_j}.F)$ artificially set equal 
to $y_1$, i.e. their twisting factor does not arise from a function defined
globally on the classifying space $\mathcal M$.  On the other hand, by
the main theorem of~\cite{KNU} we have
\begin{equation}
    e^{\sum_j\, iy_j N_j}.F = t(y) {}^e g(y) e^{\epsilon(y)}.\text{\bf r}
    \label{eq:asymp-non-split}
\end{equation}
where ${}^e g(y)\in G_{\mathbb R}$ converges to $1$ as all $y_j/y_{j+1}\to\infty$,
$\text{\bf r}\in\mathcal M$ is split over $\mathbb R$ and $\epsilon(y)$ has
a finite limit as all $y_j/y_{j+1}\to\infty$.  If this limit is non-zero,
we say that the original period map/nilpotent orbit is {\it asymptotically
non-split}.  Equivalently, by~\cite{KNU}
\begin{equation}
      \lim_{y_j/y_{j+1}\to\infty}\, t^{-1}(y) e^{\sum_j\, z_j  N_j}.F
      \label{eq:asymp-2}
\end{equation}
is not split over $\mathbb R$.  By Corollary 12.8 in~\cite{KNU}, limit
\eqref{eq:asymp-2} has the same limit if replace $e^{\sum_j\, z_j N_j}.F$
by the period map $F(z)$. In this case, it follows directly from the 
arguments of \S 12 of~\cite{KNU} and equation \eqref{eq:basic-bound} that the 
same norm estimates hold for the twisted metric, which is defined on all of 
$\mathcal M$.

\par In \S 4 of~\cite{KNU2}, the authors put a metric on a space of 
$SL_2$-orbits and obtain the norm estimates described above for a
metric which appears to be quasi-isometric to the twisted metric
studied in this paper near boundary points which correspond to 
asymptotically non-split orbits.  In this respect, our results on
the distance between the period map and its nilpotent orbit can
be considered as a complement to the results of~\cite{KNU2}.  Globally
however, the metric constructed in~\cite{KNU2} does not seem to be 
globally quasi-isometric to the twisted metric studied in this paper.

\subsection*{Reduced Period Map} Another potential application of the results 
on relative compactness considered above concerns partial compactifications of 
period domains of pure Hodge structure.  More precisely, we recall 
(cf.~\cite{KU}) that the period domain quotients $\Gamma\backslash\mathcal D$ 
can be partially compactified with respect to horizontal maps by the addition 
of boundary components which parametrize nilpotent orbits with local monodromy 
contained in the faces of a fan $\Sigma$ which satisfies certain compatibility 
conditions relative to $\Gamma$.  In the classical case of Hodge structures
of weight 1, these partial compactifications correspond to
toroidal compactification of Ash, Mumford, Rapoport and Tai.  The object
which corresponds to the limit point in Satake compactification is the 
{\it reduced limit Hodge filtration}
\begin{equation}
        \lim_{Im(z)\to\infty}\, e^{zN}.F         \label{eq:naive-1}
\end{equation}
of a nilpotent orbit of weight $k$.  
We discuss the relationship between the reduced limit Hodge filtration and the 
Satake boundary component in \S \ref{Satake} for more detail.

\par By Lemma $(3.12)$ in~\cite{CKS}, this filtration always takes values in 
the topological boundary $\partial\mathcal D$ of $\mathcal D$ in the compact 
dual $\check{\mathcal D}$.  This boundary is again a union of 
$G_{\mathbb R}$-orbits, and the study of such boundary strata is potentially 
very useful in the study of certain infinite dimensional representations of 
$G_{\mathbb R}$ (see \S 5 of~\cite{KP2}).

\par More precisely, Cattani, Kaplan and Schmid show that if $(N,H,N^+)$ is
the $sl_2$-triple attached to the Deligne splitting 
$$ 
       (\hat F,W(N)[-k]) = (e^{-i\delta}.F,W(N)[-k])
$$
of the limit mixed Hodge structure attached to \eqref{eq:naive-1} then
\begin{equation}
    e^{zN}.F = e^{i\delta}e^{(1/z)N^+}.\Phi,\qquad
    \Phi^p = \bigoplus_{s\leq k-p}\, I^{r,s}_{(\hat F,W(N)[-k])}
    \label{eq:naive-2}
\end{equation}

\begin{rem} The filtration $\Phi$ occuring in \eqref{eq:naive-2} is unchanged by replacing $\hat F$ by 
$e^{-\epsilon}.F$: Both $\epsilon$ and $\delta$ belong to 
$$
           \Lambda^{-1,-1} = \oplus_{a,b<0}\, \mathfrak g^{a,b}_{(F,W(N)[-k])}
$$
which preserves $\Phi$ and acts equivariantly with respect to the Deligne bigrading, i.e. 
$I^{p,q}_{(e^{\lambda}.F,W)} = e^{\lambda}.I^{p,q}_{(F,W)}$.
\end{rem}

In \S 6, we consider the several variable analog of this question.  Namely,
given a nilpotent orbit generated by $(N_1,\dots,N_r;F)$
then for any $y_1,\dots,y_r>0$ the data $(N(y),F)$ generates a 1-variable
nilpotent orbit where $N(y) = \sum_j\, y_j N_j$.  Using the several 
variable $SL_2$-orbit theorem, we obtain an asymptotic formula for
$N^+(y)$ in the associated $sl_2$-triple. We close \S 6 with a discussion
of the reduced limit map for variations of mixed Hodge structure.

\section*{Acknowledgments} The authors would like to thank Patrick
Brosnan, Phillip Griffiths, Aroldo Kaplan, Zhiqin Lu, Chikara Nakayama, 
Chris Peters and Colleen Robles for helpful discussion regarding various 
aspects of this paper.  This research was partially supported by 
Research Fund for International Young Scientists NSFC 11350110209 (Hayama) and 
NSF grant DMS 1002625 (Pearlstein).  We also acknowledge generous support from 
the IHES, Institut de Math\'{e}matiques de Jussieu, National Taiwan University 
and the Fields Institute which facilitated this work.

\section*{Notations and conventions}
Let $V$ be a $\RR$-vector space and let $W=\{W_k\}_{k\in \ZZ}$ be an increasing 
filtration of $V$. We use the following notations:
\begin{itemize}
\item $\gr^W_k$ is the quotient space $W_k/W_{k+1}$ and $\gr^{W}(V)=\bigoplus_k \gr^W_k$;
\item $\End{(V)}^W$ (resp. $\Aut{(V)}^W$) is the group of endomorphisms (resp. 
automorphisms) which preserves $W$;
\item For $X\in\End{(V)}^W$ (resp. $\Aut{(V)}^W$), $\gr^W_k(X)$ is the 
induced action of $X$ on $\gr^W_k$;
\item For $X\in \End{(V)}^W$ and $\lambda >0$, we write 
$\lambda^X=\exp{(\log{(\lambda)} X)}\in \Aut{(V)}^W$;
\item For the decreasing filtration $F=\{F^p\}_{p\in\ZZ}$ on 
$$
     V_{\CC}:=V\otimes \CC,\qquad 
     F(\gr^W_k)=\{F^p(\gr^W_k)\}_{p\in\ZZ}
$$ is the  filtration induced by $F$ on $\gr^W_{k,\CC}:=\gr^W_k\otimes\CC$.
\item Let $(V',W')$ be a pair of $\RR$-vector space and an increasing 
filtration and let $p:V\to V'$ be the $\RR$-linear map which preserves the 
filtration. $\gr(p)$ (resp. $\gr_k(p)$) is the linear map 
$\gr^{W}(V)\to \gr^{W'}(V')$ (resp. $\gr_k^{W}\to \gr_k^{W'}$).
\end{itemize}

Let $\mu=(\mu_1,\ldots ,\mu_r)\in \ZZ^r$.
We denote $\mu(j)=\mu_j$ for $1\leq j\leq r$.

\section{Metrics}
\subsection{Classifying Spaces}\label{cs} We fix a finite-dimensional
$\RR$-vector space $V$, an increasing filtration $W$ on $V$,
non-degenerate bilinear forms $\langle \; ,\;\rangle_w :\gr^W_w\times
\gr^W_w \to \RR$ for $w\in \ZZ$, which is symmetric if $w$ is even and
anti-symmetric if $w$ is odd, and non-negative integers $h^{p,q}_w$
with $h^{p,q}_{w}=h^{q,p}_{w}$ and with $h^{p,q}_w=0$ unless $p+q=w$
such that $\dim{\gr^W_{w}}=\sum_{p,q}h^{p,q}_{w}$.  Let $\M$ be the set
of all decreasing filtrations $F$ on $V_{\CC}$ for which $(W,F)$ is a
mixed Hodge structure (MHS) such that, for all $w\in \ZZ$,
$(\gr^{W}_w,F(\gr^W_w),\langle \; ,\;\rangle_w)$ is a polarized Hodge
structure with Hodge numbers $\{h^{p,q}_w\}$.  Let $\check{\M}$ be the
set of all decreasing filtrations on $V_{\CC}$ satisfying the
following two conditions:
\begin{itemize}
\item $\dim{(F^p(\gr^W_{w} )/F^{p+1}(\gr^W_{w} ))}=h^{p,w-p}_{w}$;
\item $\langle F^p(\gr^W_{w} ),F^q(\gr^W_{w} )\rangle_w=0$ for $p+q>w$.
\end{itemize}
Here $\check{\M}$ is a flag manifold including $\M$ as an open subset.
$\check{\M}$ is the called compact dual.

Let 
$$
    G_{A}=\{g\in GL{(V_A)}^W\; |\; 
           \gr^W_w{(g)}\in \Aut{(\gr^W_{w,A},\langle \; ,\;\rangle_w)}\}$$
for $A=\RR,\CC$.
Then $G_{\CC}$ acts on $\check{\M}$ transitively (\cite[(2.11.1)]{U}).


\subsection{Standard Metric} By a theorem of Deligne, a mixed Hodge
structure $(F,W)$ on $V_{\CC}$ determines a unique, functorial bigrading
\begin{equation}
  V_{\CC}=\bigoplus I^{p,q}          \label{eq:bigrading}
\end{equation} 
such that 
\begin{itemize}
\item[(a)] $F^p = \oplus_{a\geq p}\, I^{a,b}$;
\item[(b)] $W_k = \oplus_{a+b\leq k}\, I^{a,b}$;
\item[(c)]
$
  I^{p,q}
   =\bar{I}^{q,p} \pmod{\oplus_{r<p,s<q}I^{r,s}}
$
\end{itemize}
A MHS $(F,W)$ is called $\RR$-split if $I^{p,q} =\bar{I}^{q,p}$.  

\par A grading of $W$ corresponds to semisimple endomorphism $Y$ of 
$V_{\CC}$ such that $W_k$ is the direct sum of $W_{k-1}$ and the
$k$-eigenspace of $Y$ for each index $k$.  In particular, a mixed Hodge 
structure $(F,W)$ determines an associated grading $Y_{(F,W)}$ of $W$
which acts as multiplication by $p+q$ on $I^{p,q}$.

\begin{dfn}\cite{K,P1} Let $(F,W)$ be a graded-polarized mixed Hodge structure
with underlying complex vector space $V_{\CC}$.  Then, the associated mixed
Hodge metric $h$ is the unique hermitian inner product on $V_{\CC}$ which 
makes the associated bigrading \eqref{eq:bigrading} orthogonal and satisfies
\begin{equation}
     h(u,v)=i^{p-q}\langle \gr_{p+q}^W(u),\gr_{p+q}^W(\bar v)\rangle_{p+q}
     \label{eq:mixed-Hodge-metric}
\end{equation}
for $u,v\in I^{p,q}$.  The associated norm will be denoted $\|*\|$.
\end{dfn}

\par Suppose now that $\M$ is a classifying space of graded-polarized mixed
Hodge structures with underlying complex vector space $V_{\CC}$. Then, mixed
Hodge metric endows the trivial bundle $V\times\M\to\M$ with a hermitian
structure.  In particular, using the identification of $\M$ with a submanifold
of a flag variety, we obtain an associated hermitian metric on $\M$ which
we shall call the standard metric on $\M$.  

\begin{lem}\cite{P1} Let $\mathfrak{g}_{A}=\Lie{G_{A}}$ ($A=\RR,\CC$).  Then, by
the functoriality of the above constructions, a point $F\in\M$ induces
a mixed Hodge structure on $\mathfrak{g}_{\CC}$ and hence determines
an associated decomposition
$$
    \mathfrak{g}_{\CC}=\bigoplus_{p,q} \mathfrak{g}^{p,q}
$$
as in \eqref{eq:bigrading}. 
Since $G_{\CC}$ preserves the weight filtration, $\mathfrak{g}^{p,q}=0$ if $p+q>0$.
In particular, the Lie algebra of the
stabilizer of $F\in\M$ is 
$$
      \mathfrak{g}_{\CC}^F = \bigoplus_{p\geq 0}\, \mathfrak{g}^{p,q}
$$
and hence the subalgebra  
$$
         q_{F}=\bigoplus_{p<0,p+q\leq 0}\mathfrak{g}^{p,q}
$$
is a vector space complement to $\mathfrak{g}_{\CC}^F$ in $\mathfrak{g}_{\CC}$.
Accordingly, we may identify $q_{F}$ with $T_F(\M)$ via the map
$$
      v\mapsto \left.\frac{d}{dt}e^{tv}.F\right|_{t=0}
$$
Under this identification, the standard metric on $\M$ is given by the formula
$$ 
      h_F(\alpha,\beta)=\mathrm{Tr}{(\alpha\beta^*)}
$$
for $\alpha,\beta \in q_{F}$ where $\alpha,\beta$ in the right hand side are 
represented as matrices with respect to the unitary basis for $h_{F}$ and $\beta^*$ is the conjugate transpose in this basis. 
If we pass from $\mathrm{Lie}(G_\CC)$ to say $\mathfrak{sl}(V_\CC)$ then adjoint
maps $\mathfrak{sl}(V_\CC)^{p,q}$ to $\mathfrak{sl}(V_\CC)^{-p,-q}$.
\end{lem}

%

\begin{lem}\label{lem:isometries}\cite{K,P2} Given a point $F\in\M$, let
$$
     \Lambda^{-1,-1}  = \bigoplus_{r,s<0}\, \mathfrak g^{r,s}
$$
and  $g\in G_{\RR}\cup\exp(\Lambda^{-1,-1})$. Then,
$$
      L_{g*}:T_F(\M) \to T_{g.F}(\M)
$$
is an isometry. In particular,  the distance $d$ defined by the standard metric is a $G_{\RR}$-invariant distance 
(\cite[Theorem 2.5]{P1}).
\end{lem}

\par Let $Y$ be a real grading of $W$.  Then 
$\lambda^{-\half Y}=\exp{(-\half\log{(\lambda)}Y)}$ for $\lambda>0$ 
is an automorphism of $\M$. We then have the following Lemma:

\begin{lem}[\cite{P1} Lemma 5.7]\label{lem:P5.7}  Let $Y$ be a grading of $W$ 
defined over $\RR$, $y>0$, $\alpha\in\mathbb R$ and $F\in\M$.  Let $\|*\|$ 
denote the standard  metric~\eqref{eq:mixed-Hodge-metric}. Then, 
\begin{equation}
       \|y^{\alpha Y}(v) \|_{y^{\alpha Y}.F}=y^{(p+q)\alpha}\|v\|_{F}
       \label{eq:scale-factor}
\end{equation}
if $v\in I^{p,q}_{(F,W)}$. 
\end{lem}
\begin{proof}  Let $v\in I^{p,q}_{(F,W)}$.  Then, since $Y$ is defined
over $\RR$, 
$$
        y^{\alpha Y}(v)\in I^{p,q}_{(y^{\alpha Y}.F,W)}.
$$
Moreover, since both $Y$ and $Y_{(F,W)}$ are gradings of $W$, we can find
$u\in W_{p+q-1}$ so that $Y(v+u) = (p+q)(v+u)$.  Accordingly,
$$
\aligned
       y^{\alpha Y}(v) 
                &=  y^{\alpha Y}(v+u-u)
                 =  y^{\alpha (p+q)}(v+u) +  y^{\alpha Y}(u) \\
                &=  y^{\alpha (p+q)} v \mod W_{p+q-1}
\endaligned
$$
and hence by the definition of the standard metric,
$$ 
      \|y^{\alpha Y}(v) \|_{y^{\alpha Y}.F}= y^{\alpha (p+q)}\|v\|_{F}
$$
since $y^{\alpha Y}.F$ and $F$ induce the same Hodge structure on $Gr^W$.
\end{proof}

\begin{rem} To see that $y^{\alpha Y}.F$ remains in $\mathcal M$, observe
that it is a mixed Hodge structure by Remark \eqref{rem:splitting} since 
$y^{\alpha Y}\in GL(V_{\mathbb R})$ preserves $W$.  Moreover, since $y^{\alpha Y}$
acts by scalar multiplication on each $Gr^W_k$, it acts trivially on the
induced filtrations $F Gr^W$.
\end{rem}

\par The length of the weight filtration $W$ is the difference
\begin{equation}
      L = \min\{ k \mid W_k = V\} - \max\{k \mid W_k = 0\}
      \label{eq:length}
\end{equation}

\begin{cor}\label{cor:distortion} Under the hypothesis of 
Lemma~\eqref{lem:P5.7}
$$
       d(y^{\alpha Y}.F,y^{\alpha Y}.F') \leq y^{-\alpha (L-1)}d(F,F')
$$
if $\alpha<0$ and $y>1$.
\end{cor}
\begin{proof} Let $\gamma:[0,1]\to\mathcal M$ be a smooth curve from
$F$ to $F'$.  Then, $y^{\alpha Y}.\gamma(u)$
is a curve from $y^{\alpha Y}.F$ to $y^{\alpha Y}.F'$.  Therefore,
$$
      d(y^{\alpha Y}.F,y^{\alpha Y}.F')
      \leq\int_0^1\, 
       \|L_{y^{\alpha Y}*}\gamma'(u)\|_{y^{\alpha Y}.\gamma(u)}\,du 
$$
where $\|\cdot\|$ is the standard metric on $\mathfrak{g}_{\CC}$.
Let
$$
         \gamma'_k(u) = \sum_{p+q=k}\,\gamma'(u)^{p,q}
$$
according to the decomposition~\eqref{eq:bigrading} along $\gamma$.  Then, by
the adjoint form of equation~\eqref{eq:scale-factor},
$$ 
       \|L_{y^{\alpha Y}*}\gamma'_k(u)\|_{y^{\alpha Y}.\gamma(u)}
       = y^{k\alpha}\|\gamma'_k(u)\|_{\gamma(u)}
$$
Accordingly, since the subspaces
$
       \oplus_{p+q=k}\, \mathfrak g^{p,q}
$
are orthogonal for different values of $k$ and $\mathfrak{g}^{p,q}=0$ if $p+q\notin[-(L-1),0]$, it then follows that
$$
       \|L_{y^{\alpha Y}*}\gamma'(u)\|_{y^{\alpha Y}.\gamma(u)}
       \leq y^{-\alpha(L-1)}\|\gamma'(u)\|_{\gamma(u)}
$$
and hence
$$
      d(y^{\alpha Y}.F,y^{\alpha Y}.F')
      \leq y^{-\alpha (L-1)}\int_0^1\, \|\gamma'(u)\|_{\gamma(u)}\,du
$$
Taking the infimum over all such paths $\gamma$ yields
$$
      d(y^{\alpha Y}.F,y^{\alpha Y}.F')\leq  y^{-\alpha (L-1)}d(F,F')
$$
as required.
\end{proof}

\subsection{Twisted Metric} Let $\M$ be a classifying space of 
graded-polarized mixed Hodge structures, and
$$
       \tau:\M\to (0,\infty)
$$
be a smooth function.  Then, the associated twisted metric $|*|$ on 
$V_{\CC}\times\M\to\M$ is obtained by scaling the Hodge metric
\eqref{eq:mixed-Hodge-metric} on $I^{p,q}_{(F,W)}$ by $\tau(F)^{(p+q)/2}$.
 
\par The reason to introduce the twisted metric is to undo the scaling
effect seen in Lemma~\eqref{lem:P5.7}.  More precisely, as discussed in
the introduction, in order to obtain a relative compactness result in
the mixed case, we have to twist the period map by the inverse of 
$$
           t(y) =y_1^{-Y_0/2}\prod_{j=1}^r y_j^{-H_j/2}
$$
where $H_1,\dots,H_r\in\mathfrak{g}_{\RR}$ are the semisimple elements of 
commuting representations of $sl_2(\RR)$ on $V$ and $Y_0$
is a real grading of $W$ which commutes with $H_1,\dots,H_r$.  

\par Let $\tau$ be the twisting function defined in \eqref{eq:sl2-twist}
and $Y$ be a real grading of $W$.  Then, by Remark~\eqref{rem:splitting}
$$
\aligned
   \tau(y^{-Y/2}.F) 
    &= 1 + \sum_{p,q<0}\, \|\epsilon^{p,q}(y^{-Y/2}.F,W)\|_{y^{Y/2}.F}^{-\frac{2}{p+q}} \\
    &= 1 + \sum_{p,q<0}\, 
           \|y^{-Y/2}.\epsilon^{p,q}(F,W)\|_{y^{-Y/2}.F}^{-\frac{2}{p+q}} \\
    &= 1 + \sum_{p,q<0}\, 
       \left(y^{-\frac{p+q}{2}}\|\epsilon^{p,q}(F,W)\|_F\right)^{-\frac{2}{p+q}} \\ 
    &= 1 + \sum_{p,q<0}\, y \|\epsilon^{p,q}(F,W)\|_F^{-\frac{2}{p+q}} \\
    &= 1 + y(\tau(F)-1)
\endaligned
$$
for any $F\in\mathcal M$. Likewise, by the $G_{\mathbb R}$-invariance of the
standard metric and Remark~\eqref{rem:splitting}, it follows that
\begin{equation}
     \tau(t(y).F)
      = 1 + \sum_{p,q<0}\, y \|\epsilon^{p,q}(F,W)\|_F^{-\frac{2}{p+q}}
      = 1 + y_1(\tau(F)-1) 
      \label{eq:twisting-factor}
\end{equation}
In particular, as long as $\epsilon(F,W)$ is non-zero (i.e. $(F,W)$ is not
split over $\mathbb R$), it follows that
\begin{equation}
            \tau(t(y).F) = 1 + C y_1         \label{eq:basic-bound}
\end{equation}
 with fixed $F$ for some non-zero constant $C$ related to $F$.  Inspection of
the right hand side of equation~\eqref{eq:twisting-factor} shows that
\eqref{eq:basic-bound} holds for variable $F$ which remains in a compact
set which does not intersect the locus of points 
$\mathcal M_{\mathbb R}\subseteq\mathcal M$ for which the associated mixed
Hodge structure is split over $\mathbb R$.

\begin{rem} The series of computations outlined above hold for Deligne's
$\delta$-splitting in place of $\epsilon$, and hence the conclusions of
this paper for the twisted metric also hold for the corresponding metric
obtained by replacing $\epsilon^{p,q}$ by $\delta^{p,q}$ in the definition
of the twisted metric. Here $\exp{(\epsilon)}=\exp{(i\delta)}\exp{(-\zeta)}$ 
with a $\RR$-linear map $\zeta: V\to V$ (see \cite{KNU} for detail).
\end{rem}

\begin{cor}\label{cor:twist-bound} Let $Y$ be a grading of $W$ 
defined over $\mathbb R$ and $F\in\mathcal M$. Then, the twisted metric 
satisfies
\begin{equation}
       \frac{|t(y)(v)|_{t(y).F}}{|v|_F}
        = \left(\frac{y_1^{-1} + \tau(F) - 1}{\tau(F)}\right)^{\frac{p+q}{2}}
       \label{eq:twist-bound}
\end{equation}
if $v\in I^{p,q}_{(F,W)}$. 
\end{cor}
\begin{proof} By \eqref{eq:scale-factor} and the previous calculations,
it follows from the definition of the twisted metric that
$$
       |t(y)(v)|_{t(y).F} = (y_1^{-1} + \tau(F) - 1)^{\frac{p+q}{2}}\| v\|_F
$$
whereas $|v|_F = \tau(F)^{\frac{p+q}{2}}\| v\|_F$
\end{proof}

\par Before stating out next result, we recall the following 
from~\cite[pg 334]{C}

\begin{dfn} Let $M$ a Riemannian manifold.  A subset $X\subseteq M$
is geodesically convex if any of its points are joined by a unique 
shortest geodesic in $M$, and this geodesic lies in $X$.
\end{dfn}

\begin{thm} (Whitehead) Let $W\subset M$ be open and $x\in W$.  Then,
there is a geodesically convex open neighborhood $U\subset W$ of $x$.
\end{thm}

  
\begin{thm}\label{thm:twisted-distance} Let $d_{\tau}$ denote the Riemannian
distance associated to the twisted metric.  Assume that for 
$y_1,\dots,y_r$ sufficiently large the points $t^{-1}(y).F_1$ and
$t^{-1}(y).F_2$ are contained in a geodesically convex set $U$ that is 
contained in a compact set which does not intersect the split locus 
$\mathcal M_{\mathbb R}$.  Then, there exists a constant $K$ such that
$$
           d_{\tau}(F_1,F_2)\leq K d_{\tau}(t^{-1}(y).F_1,t^{-1}(y).F_2)
$$
\end{thm}
\begin{proof} Let $\gamma:[0,a]\to U$ be the geodesic from 
$t^{-1}(y).F_1$ to $t^{-1}(y).F_2$.  Then,
$$
  d_{\tau}(F_1,F_2) 
   \leq \int_0^a | L_{t(y)*}\gamma'(u) |_{t(y).\gamma(u)}\,du
$$
where $|*|$ denotes the twisted metric.  Let
$$
       \gamma_k'(u)  = \bigoplus_{p+q=k}\, \gamma'(u)^{p,q}
$$
according to the decomposition~\eqref{eq:bigrading} along the curve $\gamma$.
Then, by Corollary \eqref{cor:twist-bound},
$$
       | L_{t(y)*}\gamma_k'(u) |_{t(y).\gamma(u)}
       = \left(\frac{y_1^{-1} + \tau(\gamma(u)) - 1}{\tau(\gamma(u))}
         \right)^{\frac{k}{2}} |\gamma'(u)|_{\gamma(u)}
$$
In particular, as long as $\gamma(u)$ stays in a compact set which 
does not intersect the split locus (where $\tau=1$), the distortion factor in 
the previous equation can be uniformly bounded away from zero.
\end{proof}
 
\par Of course, twisting the metric in this way reduces its symmetry.
 
\begin{lem}\label{lem:twisted-sym} The twisted metric $h_{\tau}$ is
$G_{\mathbb R}$-invariant, i.e. for any $g\in G_{\mathbb R}$ and $F\in\M$
$$
      L_{g*}: T_F(\M)\to T_{g.F}(\M)
$$
is an isometry.
\end{lem}
\begin{proof} $\epsilon(g.F,W) = g.\epsilon(F,W)$, and hence
$$
       \|\epsilon^{p,q}(g.F,W)\|_{g.F} = \|\epsilon^{p,q}(F,W)\|_F
$$
The rest follows from the $G_{\mathbb R}$-invariance of the standard
metric.
\end{proof}

\subsection{Variation of mixed Hodge structure}\label{pn}
The definition of the notion of an admissible variation of mixed Hodge
structure is due to Steenbrink and Zucker in the 1-variable case~\cite{SZ}
and to Kashiwara~\cite{Ka} in the several variable case.  In particular,
locally, a variation of mixed Hodge structure is given by a holomorphic
map into a classifying space of graded-polarized mixed Hodge structure
$\mathcal M$ which satisfies Griffiths's horizontality condition
$$
       \frac{\partial}{\partial s_j}F^p(s)\subseteq F^{p-1}(s)
$$

\par As alluded to in the introduction, the admissibility of a variation
of mixed Hodge structure $\mathcal V\to\Delta^{*r}$ implies that it has
a limit Hodge filtration in analogy with the pure case.  The second 
part of admissibility is the existence of the relative weight filtration 
$M(N_j,W)$ of $W$ and each of the local monodromy logarithms.  By a theorem 
of Kashiwara, the limit Hodge filtration $F_{\infty}\in\check{\mathcal M}$ pairs 
with the relative weight filtration $M=M(N_1 + \cdots + N_r,W)$ to yield a
mixed Hodge structure for which each $N_j$ is $(-1,-1)$-morphism.  Moreover,
$$
        (e^{\sum_j\, z_j N_j}.F_{\infty},W)
$$
is an admissible nilpotent orbit.

\par Two basic references for period maps and nilpotent orbits in the 
mixed case are~\cite{KNU} and~\cite{BP1}.  Unless otherwise noted, the 
conventions in this paper conform to \S 2 of~\cite{BP1}, and we refer the 
reader two these two papers for additional background information on mixed 
period maps.

\section{Distance estimate} Let $\mathcal V$ be an admissible variation of 
mixed Hodge structure over $\Delta^{*r}$ with period map $F:U^r\to\mathcal M$ 
and associated nilpotent orbit 
\begin{equation}
       e^{\sum_j\, z_j N_j}.F_{\infty}            \label{eq:nilp-1}
\end{equation}
In this section, we prove the Schmid's original estimate between 
$F(z)$ and $e^{\sum_j\, z_j N_j}.F_{\infty}$ holds for the standard metric
on the classifying space $\mathcal M$.  More precisely, we prove
that if $L$ is the length~\eqref{eq:length} of the weight filtration
then
\begin{equation}
      d(F(z),F_{nilp}(z)) \leq K y_1^{(L-1)/2}\sum_j\, y_j^{b_j}e^{-2\pi y_j}
      \label{eq:weak-estimate}
\end{equation}
The two key steps are the derivation of equation~\eqref{eq:first-step}
and Lemma~\eqref{lem:rel-estimate} below, which both ultimately depend on
Griffiths's horizontality.

\par To begin, fix $(z_{j+1},\dots,z_r)\in U^{r-j}$.  Then, the map
$$
     (z_1,\dots,z_j) \mapsto F(z_1,\dots,z_r)
$$
arises from an admissible variation of mixed Hodge structure over $\Delta^{*j}$
with monodromy logarithms $N_1,\dots,N_j$.  Let
$F_{\infty}(z_{j+1},\dots,z_r)$ be the limit Hodge filtration of this period
map and
$$
     F_j(z_1,\dots,z_r) = e^{\sum_{k\leq j}\, z_k N_k}F_{\infty}(z_{j+1},\dots,z_r)
$$
be the associated nilpotent orbit.  

\par Following~\cite{CK}, the partial period map $F_j$ admits the following
description:  The Deligne bigrading
$$
           V_{\mathbb C} = \bigoplus_{p,q}\, I^{p,q}_{(F_{\infty},M)}
$$
of the limit mixed Hodge structure $(F_{\infty},M)$ of $\mathcal V$ induces 
a bigrading 
$$
       \mathfrak g_{\mathbb C} = \bigoplus_{a,b}\,\mathfrak g^{a,b}
$$
of $\mathfrak g_{\mathbb C}\subset V_{\mathbb C}\otimes V_{\mathbb C}^*$ such that
$$
        \mathfrak g^{a,b}(I^{p,q}_{(F_{\infty},M)})
        \subseteq I^{p+a,q+b}_{(F_{\infty},M)}
$$
By equation $(6.11)$ in~\cite{P1}, the period map has a local normal form
\begin{equation}
      F(z) = e^{\sum_j\, z_j N_j}e^{\Gamma(s)}.F_{\infty}
      \label{eq:lnf}
\end{equation}
where 
$$
      \Gamma:\Delta^r\to\mathfrak q,\qquad \Gamma(0) = 0
$$
is a holomorphic function taking values in the complement
$$
     \mathfrak q = \bigoplus_{a<0,b}\, \mathfrak g^{a,b}
$$
to the Lie algebra of the stabilizer of $F_{\infty}$ in $G_{\mathbb C}$.
For future reference, we record that $\mathfrak q$ has a natural
grading
$$
      \mathfrak q = \oplus_{a<0}\, \wp_a,\qquad 
      \wp_a = \bigoplus_b\, \mathfrak g^{a,b}
$$

\par To simplify notation, let us define $N(z) = \sum_{j=1}^r\, z_j N_j$
and set
\begin{equation} 
\label{eq:Gam_j}
          \Gamma_j(s)=\Gamma(\underbrace{0,\ldots ,0}_j,s_{j+1},\ldots ,s_n)
\end{equation}
Then, tracing through the above definition, it follows that
$$
      F_{\infty}(z_{j+1},\dots,z_r) 
      = e^{\sum_{k>j} z_k N_k}e^{\Gamma_j(s)}.F_{\infty}
$$
and hence $F_j(z) = e^{N(z)}e^{\Gamma_j(s)}.F_{\infty}$.

\begin{rem} The local normal form is only valid in some relative
neighborhood of $0$ in $\Delta^{*r}$.  Moreover, since $F_j(z)$ is constructed 
as a nilpotent orbit, technically we must further shrink $\Delta^{*r}$ in order 
get an admissible variation, i.e. $F_j(z)$ only takes values in $\mathcal M$ 
when
\begin{equation}
      \text{Im}(z_1),\dots,\text{Im}(z_r)>\alpha   \label{eq:implicit}
\end{equation}
for some constant $\alpha$.  However, since we are only interested in 
asymptotics, for simplicity of notation we shall leave~\eqref{eq:implicit}
as an implicit assumption in the remainder of this section.
\end{rem}


\begin{lem}[\cite{CK}Proposition 2.6]\label{lem:CK2.6}
For $1\leq k \leq j\leq r$, $[N_k,\Gamma_j]=0$.
\end{lem}
\begin{proof}
By the horizontality condition, 
$\frac{\partial}{\partial z_j}\log{(e^{N(z)}e^{\Gamma(s)})} \in \wp_{-1}$.
Therefore,
\begin{align}
\nonumber
\frac{\partial}{\partial z_j} &\log{(e^{N(z)}e^{\Gamma(s)})} \\
  &= e^{-\Gamma(s)}e^{-N(z)}\l(e^{N(z)}N_je^{\Gamma(s)}+e^{N(z)}2\pi i s_j
     \frac{\partial}{\partial s_j}e^{\Gamma(s)}\r)  \label{der-Gamma} \\
\nonumber
  &=e^{-\ad{\Gamma(s)}}N_j+2\pi i s_je^{-\Gamma(s)}
     \frac{\partial}{\partial s_j}e^{\Gamma(s)}\in \wp_{-1}.
\end{align}
Since the left hand side of (\ref{der-Gamma}) is holomorphic function on 
$\Delta^r$ and the right hand side of (\ref{der-Gamma}) is a complex vector 
space, the equation holds on all of $\Delta^r$.
Setting $s_j=0$, we have $e^{-\ad{\Gamma(s)}}N_j\in\wp_{-1}$.
Since $N_j\in\wp_{-1}$ and $\Gamma(s)\in\mathfrak q_F$, this is possibly only 
if $[N_j,\Gamma|_{s_j=0}]=0$.
\end{proof}

\par The basic strategy of the proof of Schmid's distance estimate in
the mixed case is observe that $F_0 = F(z)$ whereas $F_r(z)$ is the
associated nilpotent orbit.  Therefore, we have the triangle inequality
\begin{equation}
     d(F_0,F_r) \leq d(F_0,F_1) + \cdots + d(F_{r-1},F_r)
                                \label{eq:triangle}
\end{equation}
We define
$$
        e^{\Gamma^j(s)} = e^{\Gamma_{j-1}(s)}e^{-\Gamma_j(s)}
$$
and then $\Gamma^j(s)$ is a holomorphic function which is divisible by $s_j$.

\par To continue, we observe that if we view $F_j(z)$ as a period map
in the variables $(z_1,\dots,z_r)$ then it has the same associated
nilpotent orbit~\eqref{eq:nilp-1} as $F(z)$, and hence the same
associated family of semisimple operators
$$
        t(y) = y_1^{-Y^0/2}\Pi_j\, y_j^{-H_j/2}
$$
Let $N(x) = \sum_j\, x_j N_j$ where $z_j = x_j + \sqrt{-1} y_j$.

\begin{lem}\label{lem:distortion-2}
Define
$$
     \tilde F_j(z) = t^{-1}(y) e^{N(iy)}e^{\Gamma_j(s)}F_{\infty}.
$$
Then, $d(F_{j-1}(z),F_j(z))\leq y_1^{(L-1)/2}d(\tilde F_{j-1}(z),\tilde F_j(z))$.
\end{lem}
\begin{proof} We have
$$
\aligned
   d(F_{j-1}(z),F_j(z))
   &= d(e^{N(x)}e^{iN(y)}e^{\Gamma_{j-1}(s)}.F_{\infty},
        e^{N(x)}e^{iN(y)}e^{\Gamma_j(s)}.F_{\infty}) \\
   &= d(e^{iN(y)}e^{\Gamma_{j-1}(s)}.F_{\infty},
        e^{iN(y)}e^{\Gamma_j(s)}.F_{\infty})         \\
   &= d(t(y)t^{-1}(y)e^{iN(y)}e^{\Gamma_{j-1}(s)}.F_{\infty},
        t(y)t^{-1}(y)e^{iN(y)}e^{\Gamma_j(s)}.F_{\infty}) \\
   &= d(t(y)\tilde F_{j-1}(z), t(y)\tilde F_j(z))  \\
   &= d(y_1^{-Y_0/2}\Pi_j\, y_j^{-H_j/2}.\tilde F_{j-1}(z), 
        y_1^{-Y_0/2}\Pi_j\, y_j^{-H_j/2}.\tilde F_j(z)) \\
   &= d(y_1^{-Y_0/2}.\tilde F_{j-1}(z), 
        y_1^{-Y_0/2}.\tilde F_j(z))  \\
   &\leq y_1^{(L-1)/2}d(\tilde F_{j-1}(z),\tilde F_j(z))
\endaligned
$$
where the last step is justified by Corollary~\eqref{cor:distortion}.
All other transformations are either based on substitutions or the
fact that the group $G_{\mathbb R}$ acts by isometries with respect
to the standard metric.
\end{proof}

\par Combining the previous corollary with~\eqref{eq:triangle}, it follows
that
\begin{equation}
     d(F_0(z),F_r(z))
     \leq y_1^{(L-1)/2}\sum_j\, d(\tilde F_j(z),\tilde F_{j-1}(z))
     \label{eq:first-step}
\end{equation}
and hence it remains to prove that 
$$
       d(\tilde F_{j-1}(z),\tilde F_j(z)) \leq K y_j^{b_j} e^{-2\pi y_j}
$$
for suitable constants $K$ and $b_j$.  For this we state two preliminary
results:

\begin{lem}\label{lem:P-function} (\cite{BP1} \S 6, \cite{KNU} \S 10, 
\cite{CK} \S 4) Let $(e^{\sum_j z_j N_j}.F,W)$ be an admissible nilpotent
orbit.  Then
$$
      Ad(t^{-1}(y))e^{N(iy)} = e^{P(t)}
$$
where $P(t)$ is a polynomial in non-negative half-integer powers of 
$t_j = y_{j+1}/y_j$ where we formally set $y_{r+1} = 1$.
\end{lem}

\par A short calculation shows that
$$
      t(y) = \Pi_{j=1}^r\, t_j^{Y_j}
$$
where $t_j = y_{j+1}/y_j$ where formally $y_{r+1} = 1$, and
$$
    Y_1 = Y_0 + H_1,\cdots, Y_r = Y_0 + H_1 + \cdots + H_r
$$
are commuting semisimple endomorphisms.    Let
$$
      \Gamma^{j}=\sum_{\mu\in\ZZ^r}\Gamma^{j,[\mu]}
$$
relative to $Y^1,\dots,Y^r$, i.e. $\Ad{(t(y)^{-1})}$  acts on $\Gamma^{j,[\mu]}$ 
as multiplication by  
$$
         \prod_{j=1}^r t_j^{\mu(j)/2}.
$$

\begin{lem}\label{lem:vanishing} $\Gamma^{j,[\mu]} = 0$ if $\mu(k)>0$ for 
some $1\leq k\leq j-1$.
\end{lem}
\begin{proof} This follows from relative compactness of $\tilde F_j$
since we can make the variables $z_1,\dots,z_{j-1}$ arbitrarily large,
independent of $(s_j,\dots,s_r)$.  For details, see equation (7.32)
in~\cite{BP1}.
\end{proof}

\begin{rem} Another way of looking at the previous Lemma is that since
we know that $\Gamma_j$ commutes with $N_k$ for $k\leq j$, we expect
$\Gamma^j$ to have only non-positive weights with respect to 
$H_1,\dots,H_{j-1}$.
\end{rem}

\begin{lem}\label{lem:rel-estimate} There exist constants $K$ and $b_j$
such that 
$$
     d(\tilde F_{j-1}(z),\tilde F_j(z)) \leq K y_j^{b_j} e^{-2\pi y_j}
$$
for $\text{Im}(z_1),\dots,\text{Im}(z_r)$ sufficiently large.  
\end{lem}
\begin{proof} Since $e^{\Gamma_{j-1}}=e^{\Gamma^{j}}e^{\Gamma_{j}}$, we have
\begin{align*}
\tilde{F}_{j-1}(z)&=t(y)^{-1}e^{iN(y)}e^{\Gamma_{j-1}(s)}F_{\infty}\\
&=t(y)^{-1}e^{iN(y)}e^{\Gamma^{j}(s)}e^{\Gamma_{j}(s)}F_{\infty}\\
&=t(y)^{-1}e^{iN(y)}e^{\Gamma^{j}(s)}e^{-iN(y)}e^{iN(y)}e^{\Gamma_{j}(s)}F_{\infty}\\
&=\Ad{(t(y)^{-1})}(e^{iN(y)})
   \Ad{(t(y)^{-1})}(e^{\Gamma^{j}(s)})\Ad{(t(y)^{-1})}(e^{-iN(y)})\tilde{F}_{j}(z).
\end{align*}
Accordingly,
$$
   \tilde{F}_{j-1}(z)= \exp(\Ad(e^{P(y)})u^j(z)).\tilde{F}_j(z)
$$
where $u^j(z)=\Ad{(t(y)^{-1})}(\Gamma^j(s))$.  In particular, by 
Theorem~\eqref{thm:rel-compact}, $\tilde F_j(z)$ assumes values in
a relatively compact subset of $\mathcal M$.  Likewise, by
Lemma~\eqref{lem:P-function}, under these hypotheses, $e^{P(y)}$ takes
values in a compact subset of $G_{\mathbb C}$.  Finally, by 
Lemma~\eqref{lem:vanishing} and the fact that $s_j | \Gamma_j$, it
follows that $u^j(z)$ satisfies a bound of the form
$$
          |u^j(z)| \leq c y_j^b |s_j|
$$
(for some fixed norm) since $y_j\geq y_{j+1}\geq\cdots\geq y_r$ on $I'$.  
Combining these observations, it then follows that
$$
         d(\tilde F_{j-1}(z),\tilde F_j(z))\leq K y_j^{b_j}e^{-2\pi y_j}
$$
\end{proof}

\begin{rem} In Schmid's estimate the factor $\Pi_j\,\text{Im}(z_j)$ appears
instead of a leading power of $y_1$.  However, since we can always reorder
to obtain $y_1\geq y_2\geq ..\geq y_r$, this is really just a fully symmetric
version of our result.  
\end{rem}

\section{Strong estimates} A variation of mixed Hodge structure is said
to satisfy the strong distance estimate with respect to a metric on 
$\M$ if there is a constant $K$ such
\begin{equation}
         d(F(z),F_{nilp}(z)) \leq K\sum_j y_j^{\beta_j}e^{-2\pi y_j}
         \label{eq:strong-estimate}
\end{equation}
for all $z\in I'$ with the imaginary part of $z$ sufficiently large.
 
\subsection{Pure Case} For variations of pure Hodge structure, $L=1$, and we 
recover the strong distance estimate of~\cite{CKS} 
from~\eqref{eq:weak-estimate}.

\subsection{Twisted Metric} The distance between the period
map and the nilpotent orbit satisfies the strong distance estimate
with respect to the twisted metric, provided that we also restrict
the real parts of $z_1,\dots,z_r$ to a compact set, and assume that
the associated nilpotent orbit is asymptotically non-split in the
sense of \eqref{eq:asymp-non-split}.  To this end, we note that
since \eqref{eq:asymp-2} has the same limit when we replace the
nilpotent orbit by the period map, the twists of $F_{j-1}(z)$ and $F_j(z)$ 
are asymptotically non-split and take values in a geodesically convex 
neighborhood as $y_j/y_{j+1}\to\infty$.

\par The proof proceeds as in the previous section up to 
Lemma~\eqref{lem:distortion-2}.  At this juncture, use 
Theorem~\eqref{thm:twisted-distance} to write
$$
     d_{\tau}(F_{j-1}(z),F_j(z))\leq
     d_{\tau}(t^{-1}(y).F_{j-1}(z),t^{-1}(y)F_j(z))
$$
However, this is not exactly what appears in our relative compactness
result, since
$$
       t^{-1}(y)F_j(z) = (\Ad(t^{-1}(y))e^{N(x)})\tilde F_j(z)
$$
but as discussed in the proof of Lemma (7.26) in~\cite{BP1}, 
$$
       \Ad(t^{-1}(y))e^{N(x)} = e^{Q(x,t)}
$$
where $Q(x,t)$ is a polynomial in non-negative half-integral
powers of $t_j = y_{j+1}/y_j$ with coefficients which are polynomials
in $x_1,\dots,x_r$.  Accordingly, when the variables $x_j$ are restricted
to a compact set, we still have a relative compactness result.  The
rest of the proof proceeds as in the previous section.  Consequently,
we obtain the strong distance estimate since we never multiply by 
$y_1^{(L-1)/2}$.



\subsection{Unipotent Variations} Let $X$ be a smooth complex 
algebraic variety and $J_x$ denote the augmentation ideal of the
group ring $\mathbb Z\pi_1(X,x)$, i.e.
$$
     J_x = \{\, \sum_{g\in\pi_1}\, a_g g \mid \sum_{g\in\pi_1}\, a_g = 0\,\}
$$
Then, the quotients $\mathbb Z\pi_1(X,x)/J_x^{\ell}$ carry functorial mixed 
Hodge structures which patch together to form an admissible variation of 
mixed Hodge structure $\mathcal J$ over $X$ with monodromy representation 
given by the natural action of $\pi_1(X,x)$ on 
$\mathbb Z\pi_1(X,x)/J_x^{\ell}$.  This is called a tautological variation in 
\cite{HZ}. The weight graded-quotients of $\mathcal J$ are locally constant.
 
\par A variation of mixed Hodge structure $\mathcal V$ over $X$ is said to
be {\it unipotent} if the induced variations of pure Hodge structure on the
weight graded-quotients are constant.  By Theorem 1.6 of~\cite{HZ}, there is 
an equivalence of categories between mixed Hodge theoretic representations of 
$\mathbb Z\pi_1(X,x)/J_x^{\ell+1}$ and unipotent variations of mixed Hodge
structure with index of unipotency $\leq \ell$.
 
\par Suppose now that we have an (admissible) unipotent variation 
$\mathcal V$ over the punctured disk with monodromy operator $T$.
Then, since $\Gr^{\mathcal W}$ is constant, and $T$ is quasi-unipotent
in the usual sense, it follows that $T=id$ on $\Gr^{\mathcal W}$, and
hence $T = e^N$ where $N=0$ on $\Gr^{\mathcal W}$. Admissibility then forces:
\begin{itemize}
\item[(a)] $N(W_k)\subseteq W_{k-2}$;
\item[(b)] $M(N,W) = W$;
\end{itemize}
(see the paragraph above \cite[Theorem 1.6]{HZ}).   Thus, for a unipotent 
variation the Hodge filtration may degenerate, but the weight filtration does 
not.

\begin{thm} If $\mathcal V\to\Delta^{*r}$ is a unipotent variation
of mixed Hodge structure then period map and the nilpotent orbit
satisfy the strong distance estimate~\eqref{eq:strong-estimate}
with respect to the standard metric.
\end{thm}

\par We start with the estimate
$$
      d(F_0(z),F_1(z))\leq K y_1^{b_1}e^{-2\pi y_1}
$$
which we prove as in the previous section. It remains therefore to bound 
$$
           d(F_k(z),F_{k+1}(z))
$$ 
for $k=1,\ldots,n-1$.  

\begin{lem}
$N_j$ is a $(-1,-1)$-morphism of $(F_k(z),W)$ for $1\leq j\leq k$.
\end{lem}
\begin{proof}
By Lemma \ref{lem:CK2.6}, we know that
$$
       [N_j,\Gamma_k(s)] = 0
$$
for $j=1,\dots,k$, and hence
\begin{align}\label{N-trans}
         N_jF^p_k(z)
          &= N_j(e^{\sum_{\ell}\, z_{\ell} N_{\ell}}e^{\Gamma_k(s)}.F^p_{\infty})  
          = e^{\sum_{\ell}\, z_{\ell} N_{\ell}}e^{\Gamma_k(s)}.N_jF^p_{\infty}\\    
         & \subseteq e^{\sum_{\ell}\, z_{\ell} N_{\ell}}e^{\Gamma_k(s)}.F^{p-1}_{\infty}
          = F^{p-1}_k(z)\nonumber
\end{align}
On the other hand, $N_j$ is real and $N_j(W_{\ell})\subset W_{\ell-2}$.
Taken together, these three facts imply that if $j\leq k$ then $N_j$ 
is a $(-1,-1)$-morphism of $(F_k(z),W)$.
\end{proof}

\par Regarding the local normal form of $F(z)$, a short argument
shows that since the graded variations of Hodge structure are constant,
and the monodromy logarithms act trivially on $Gr^W$ so do the functions
$\Gamma_j$ constructed in the previous section.  In particular, if 
$A$ acts trivially on $Gr^W$ then $e^A:\mathcal M\to\mathcal M$.

\par To compactify notation, write
$$
         z = (u,v),\qquad u=(z_1,\dots,z_k),\qquad v = (z_{k+1},\dots,z_r).
$$
and observe that $F_k(0,v)=e^{N(0,v)}e^{\Gamma_{k}(0,v)}F_{\infty}\in \M$ by the
remarks of the previous paragraph.  Let $\gamma(t,v)$ be
the path
$$
         \gamma(t,v) 
          = e^{N(0,v)}e^{t\Gamma^k}e^{\Gamma_{k}}.F_{\infty}
$$
from $e^{-N(u,0)}F_{k}(z)$ to $e^{-N(u,0)}F_{k-1}(z)$.

\begin{lem} $N_1,\dots,N_k$ are $(-1,-1)$-morphisms along
$\gamma(t,v)$.
\end{lem}
\begin{proof} Since both $\Gamma_k$ and $\Gamma_{k+1}$ commute with $N_j$ for 
$j=1,\dots,k$ so does $\Gamma^k$. 
Arguing as in equation (\ref{N-trans}) it then follows that $N_1,\dots,N_k$ 
are $(-1,-1)$-morphisms along $\gamma(t,v)$.
\end{proof}

\begin{lem}
Given $k>0$ let $L(v)$ denote the length of the curve 
$\gamma(t,v)$ with respect to the standard metric. Then, 
$d(F_k(z),F_{k+1}(z))\leq L(v)$.
\end{lem}
\begin{proof} 
By definition, $d(F_k(z),F_{k+1}(z))$ is less than or equal
to the length $\tilde L$ of the path $e^{N(u,0)}\gamma(t,v)$ from $F_k(z)$ 
to $F_{k+1}(z)$.  Now we have
$$
        \tilde L = \int_0^1\, 
                   \left\| \frac{d}{dt} e^{N(u)}\gamma(t,v)
                   \right\|_{e^{N(u,0)}\gamma(t,v)}\, dt
                 = \int_0^1\, 
                   \left\| e^{N(u,0)}\frac{d}{dt}\gamma(t,v)
                   \right\|_{e^{N(u,0)}\gamma(t,v)}\, dt.
$$
By Lemma~\eqref{lem:isometries}, the fact that $N(u,0)$ is a $(-1,-1)$-morphism
along $\gamma(t,v)$ implies that 
$$
       \tilde L = \int_0^1\, 
                   \left\| e^{N(u)}\frac{d}{dt}\gamma(t,v)
                   \right\|_{e^{N(u)}\gamma(t,v)}\, dt
                = \int_0^1\, 
                   \left\| \frac{d}{dt}\gamma(t,v)
                   \right\|_{\gamma(t,v)}\, dt = L(v)
$$
and hence $d(F_k(z),F_{k+1}(z))\leq L(v)$ as required.
\end{proof}

\par It remains to estimate $L(v)$:

\begin{lem} $L(v) \leq K y_{k+1}^{b} e^{-2\pi y_{k+1}}$ for suitable non-negative
constants $K$ and $b$.
\end{lem}
\begin{proof} We begin with the observation that
$$
   F(0,v) = e^{\sum_{j>k}\, z_j N_j}e^{\Gamma_{k+1}}.F_{\infty}
$$
is the period map of a variation of mixed Hodge structure over
$\Delta^{r-k}$ in the variables $(s_{k+1},\dots,s_r)$.  Let
$$
    F_{\nilp}(0,v) =  e^{\sum_{j>k}\, z_j N_j}.F_{\infty}
$$
be the associated nilpotent orbit.  Let $t(y)$ $(y=(y_k,\ldots ,y_r))$ be as 
in \eqref{eq:t-y}) constructed from $F_{\nilp}(0,v)$. 
Then, by Theorem \eqref{thm:rel-compact}, the image of the map 
$$
            \tilde{F}(0,v) = t^{-1}(y).F(0,v)
$$
will be a relative compact subset of $\M$ for $v$ in an appropriate $I'$
(with the real component of $v$ restricted to a compact set).
Therefore,
$$
\aligned 
      L(v) &= \int_0^1\, 
                    \left\| \frac{d}{dw}\gamma(w,v)
                    \right\|_{\gamma(w,v)}\,dw  \\
           &= \int_0^1\, 
                \left\| 
            \frac{d}{dw}e^{N(0,v)}e^{w\Gamma^k}e^{\Gamma_{k+1}}.F_{\infty}
            \right\|_{e^{N(0,v)}e^{w\Gamma^k}
                    e^{\Gamma_{k+1}}.F_{\infty}}\, dw \\ 
           &= \int_0^1\, 
              \left\|\frac{d}{dw}(\text{Ad}(e^{N(0,v)})e^{w\Gamma^k}).F(0,v)
                \right\|_{(\text{Ad}(e^{N(0,v)})e^{w\Gamma^k}).F(0,v)}\,dw   
\endaligned
$$
To continue, let 
$$
       A(w,v) = \Ad(t^{-1}(y)e^{N(0,v)})e^{w\Gamma^k})
$$
Then, by the above computations,
$$
 \aligned
       L(v) &= \int_0^1\, 
                    \left\| t(y) 
                    \frac{d}{dw} A(w,v).\tilde{F}(0,v)
                    \right\|_{t(y) A(w,v).\tilde{F}(0,v)} dw
 \endaligned
$$
Finally, since $s_{k+1}$ divides $\Gamma^k$ and $t(y)$ and $e^{N(v)}$ can
be bounded by polynomials in half-integral powers of $y_{k+1}$ as in 
\S 3, it follows that 
$$
     \|\frac{d}{dw} A(w,v)\|_{\tilde{F}(0,v)} \leq C y_{k+1}^be^{-2\pi y_{k+1}}
$$
for suitable constants $C$ and $b$.  In particular, since $\tilde{F}(0,v)$ 
takes values in a relatively compact set, and the distortion introduced
by $t(y)$ is again at worst a non-negative, half-integral power of $y_{k+1}$, 
it follows that
$$
          L(v) \leq C y_{k+1}^b e^{-2\pi y_{k+1}}
$$
\end{proof}

\section{Archimedean Heights} In this section, we prove 
Theorem~\eqref{thm:loc-int}.  We begin with a discussion of the
biextension metric.

\par Let $(F,W)$ be a mixed Hodge structure on $V$ with weight
graded quotients
\begin{equation}
         Gr^W_{-2} \cong \mathbb Z(1),\qquad Gr^W_{-1} = H,\qquad 
         Gr^W_0 \cong \mathbb Z(0)                 \label{eq:biext-Gr}
\end{equation}
Assume that $V$ has lattice $V_{\mathbb Z}$ such that 
$W_a V_{\mathbb Z}/W_b V_{\mathbb Z}$ is torsion free $(b < a)$ and $H =
Gr^W_{-1}$ 
is a polarizable Hodge structure of weight $-1$ of arbitrary level.
\medskip

 For $(F,W)$ as in~\eqref{eq:biext-Gr}, fix the isomorphisms 
$Gr^W_{-2j}\cong\mathbb Z(j)$ and  the extension classes
\begin{equation}
\begin{gathered}
        0 \to W_{-2} \to W_{-1} \to Gr^W_{-1}\to 0 \\
        0 \to Gr^W_{-1}\to W_0/W_{-2} \to Gr^W_0 \to 0
\end{gathered}\label{eq:ext-data}
\end{equation}
Let $1$ and $1^{\vee}$ be the generators of $\mathbb Z(0)$ and $\mathbb Z(1)$, and
$\mu:V\to W_{-2}$ be the linear map defined by the following 
properties:
\begin{itemize}
\item[(i)]  $\mu$ annihilates $W_{-1}$; 
\item[(ii)] $\mu(v) = 1^{\vee}$ for any lifting $v\in V$ of $1\in Gr^W_0$;
\item[(iii)] $\mu$ is in the center of $\mathfrak g_{\mathbb C}=\Lie{G_{\CC}}$;
\end{itemize}
Let $\tilde B$ be the set of all mixed Hodge structures 
$(\tilde F,W)$ on $V$  with the fixed graded pieces \eqref{eq:biext-Gr} and the extension data \eqref{eq:ext-data},
and $B$ denote the quotient of $\tilde B$ by the equivalence
relation
$$
        (\tilde F,W) \equiv (e^{a\mu}.\tilde F,W),\qquad a\in\mathbb Z
$$
Then, $\mathbb C^*$ acts simply transitively on $B$ via the rule
\begin{equation}
       [t.F] = [e^{\frac{1}{2\pi i}\log(t)\mu}.F] \label{eq:Gm-action}
\end{equation}
\medskip

Let $Y_{(F,W)}$ be the semisimple endomorphism of 
$V_{\mathbb C}$ which acts as multiplication by $p+q$ on $I^{p,q}_{(F,W)}$.  Then, 
$$
      \delta_{(F,W)}\in W_{-2} End(V_{\mathbb C})
$$
is defined (see Prop. (2.20) in [CKS]) by the equation
$$
      \bar Y_{(F,W)} = Y_{(F,W)} - 4 i\delta_{(F,W)}
$$
(due to the short length of $W$), and hence
\begin{equation}
  \delta_{(F,W)} = \frac{1}{4i}(Y_{(F,W)} - \bar Y_{(F,W)})\label{eq:delta-eq}
\end{equation}
Moreover, since $\delta$ depends only on the isomorphism class of the 
underlying $\mathbb R$-MHS, $\delta$ descends to $B$.  

\par Regarding the $\mathbb C^*$ action \eqref{eq:Gm-action}, we have:
\begin{eqnarray*}
      \delta_{(t.F,W)} &=& \frac{1}{4i}(Y_{(t.F,W)} - \bar Y_{(t.F,W)}) \\
                       &=& \frac{1}{4i}(
                       (Y_{(F,W)}+\frac{1}{\pi i}\log(t)\mu) - 
                       \overline{(Y_{(F,W)}+\frac{1}{\pi i}\log(t)\mu)}) \\
                       &=& \delta_{(F,W)} - \frac{1}{2\pi}\log|t|\mu
\end{eqnarray*}
Accordingly, if for any equivalence class $[F]\in B$, we define
\begin{equation}
      |[F]| = e^{-2\pi\delta_{(F,W)}/\mu} \in\CC            \label{eq:metric}
\end{equation}
it follows that
$$
      |t.[F]| = e^{-2\pi(\delta_{(F,W)}-\frac{1}{2\pi}\log|t|\mu)/\mu} 
              = e^{-2\pi\delta_{(F,W)}/\mu + \log|t|}                  
              = |t| |[F]|
$$

\par Suppose now that $S$ is a Zariski open subset of a complex
manifold $\bar S$.  Let $\mathcal V_{\mathbb Z}$ be a local system of torsion 
free $\mathbb Z$-modules of finite rank equipped with a weight filtration $W$ 
which satisfies the analog of~\eqref{eq:biext-Gr}, 
e.g.~$Gr^W_{-2j} \cong \mathbb Z(j)$ and 
$\mathcal H = Gr^W_{-1}$ is a polarizable variation of Hodge structure 
of weight $-1$.  Let $\nu_A$ and $\nu_B$ be admissible normal functions
on $S$ which correspond to extensions
\begin{equation}
\begin{gathered}
        0 \to \mathbb Z(1)\to W_{-1}\to \mathcal H\to 0 \\
        0 \to \mathcal H\to W_0/W_{-2} \to \mathbb Z(0)\to 0
\end{gathered}\label{eq:nf-ext-data}
\end{equation}
as in \eqref{eq:ext-data}.  

\begin{lem} For any sufficiently small open set
$U$ of $S$ in the analytic topology, there exists an admissible
variation of mixed Hodge structure $\mathcal V$ over $U$ with
underlying weight filtration $W_*\mathcal V_{\mathbb Z}$ and extension data
\eqref{eq:nf-ext-data}.  
\end{lem}

\par Accordingly, there exists a unique $\mathbb C$-line bundle $\mathcal B$ 
over $S$ having such VMHS as local generators.  Indeed, the fibers of 
$\mathcal B$ are just
the sets $B$ defined above.  Since the $\mathbb C^*$-action \eqref{eq:Gm-action}
is simply transitive and the Hodge filtration of a VMHS is holomorphic with
respect to the flat structure $\mathcal V_{\mathbb Z}$ it follows that if 
$U_i$ and $U_j$ are open sets which support VMHS $\mathcal V_i$ and 
$\mathcal V_j$  which  satisfy the hypothesis of the Lemma then 
$\mathcal V_i = f_{ij}.\mathcal V_j$ for some element 
$f_{ij}\in\mathcal O^*(U_i\cap U_j)$.    Furthermore, since
\eqref{eq:Gm-action} is a simply transitive group action, $\{f_{ij}\}$ 
is a cocycle.

\begin{thm} Let $\mathcal V\to\Delta^{*r}$ be an admissible variation of mixed
Hodge structure of biextension type and
$$
       |\mathcal V| = e^{-\phi}
$$
Then, $\phi\in L^1_{loc}(\Delta^r)$.  
\end{thm}
\begin{proof} Let $F:U^r\to\mathcal M$ denote the lift of the period map
of $\mathcal V$ to the product of upper half-planes, and $t(y)$ be the
associated family of semsimple automorphisms~\eqref{eq:t-y}.  Then,
since $\mu$ is central in $\mathfrak g_{\mathbb C}$  and
$$
     F(z_1 + a_1,\dots,z_r + a_r) = e^{\sum_j\, a_j N_j}.F(z)
$$
for $a_1,\dots,a_r\in\mathbb Z$ and $\mu$ is central in $\mathfrak g_{\mathbb C}$
it follows from \eqref{eq:delta-eq} that 
 $$
       \delta_{(F(z_1 + a_1,\dots,z_r + a_r),W)}
       = \delta_{(F(z),W)}
$$ 
Likewise, $\delta_{(e^{-N(x)}F(z),W)} = \delta_{(F(z),W)}$ since $N(x)$ is
real.  By equation~\eqref{eq:metric}, it then follows that
$$
       |\mathcal V|_s = e^{-2\pi\delta_{(F(z),W)}/\mu}
$$
for any lifting of $s\in\Delta^{*r}$ to $z\in U^r$, i.e.
$\phi = 2\pi\delta_{(F(z),W)}/\mu$.  Returning to \eqref{eq:delta-eq}, 
we see that since $t(y)$ is real and preserves the weight filtration,
$$
\aligned
       \delta_{(F(z),W)} 
       &= \delta_{(e^{-N(x)}.F(z),W)}                  \\
       &= \delta_{(t(y)t^{-1}(y)e^{-N(x)}.F(z),W)}    
        = t(y)\delta_{(t^{-1}(y)e^{-N(x)}.F(z),W)}
\endaligned
$$
As above, the factors $y_j^{-H_j/2}$ appearing $t(y)$ act trivially
on $\delta$, and hence
$$
\aligned
       \delta_{(F(z),W)} 
       &= y_1^{-Y_0/2}.\delta_{(t^{-1}(y)e^{-N(x)}.F(z),W)}      \\
       &= y_1 \delta_{(t^{-1}(y)e^{-N(x)}.F(z),W)}
        = -\frac{1}{2\pi}\log|s_1| 
           \delta_{(t^{-1}(y)e^{-N(x)}.F(z),W)}
\endaligned
$$
since $\delta$ lowers the weight filtration by $2$, which is the
maximum possible.  Combining the above, it then follows that
\begin{equation}
      \phi = -\log|s_1|\delta_{(t^{-1}(y)e^{-N(x)}.F(z),W)}/\mu
           \label{eq:phi-compact}
\end{equation}
which is locally integrable since $t^{-1}(y)e^{-N(x)}.F(z)$ takes values
in a relatively compact subset of $\mathcal M$, so the corresponding
$\delta$-values are bounded, and $\log|s_1|$ is locally integrable.
\end{proof}

\section{Reduced Limit Filtrations}  In this section, we considered
the reduced limit filtration in the case of a several variable degeneration
of Hodge structure $\mathcal V\to\Delta^{*r}$, and show this filtration 
remains the same even along sequences 
$
         s(m) = (s_1(m),\dots,s_r(m))
$
for which the components $s_j(m)$ go to zero at radically different rates.

\subsection{Complement to SL(2)-orbit theorem}
Let $(N_1,\ldots ,N_r,F)$ be a pure nilpotent orbit of weight $w$, and let 
$(\rho ,\phi)$ be the associated $SL_2$-orbit.
We denote by $\hat{N}_j^{\pm}$ (resp. $H_j$) the image of the $j$-th 
$\mathbf{n}_{\pm}$ (resp. $\mathbf{h}$) where 
$(\mathbf{n}_-,\mathbf{h},\mathbf{n}_+)$ are the standard generators of 
$sl_2(\RR)$, and we write $H_{(r)}=H_1+\cdots +H_r$.
By the bracket relations of the $sl_2$-triples, we have
$$[\sum_{j=1}^{r}\frac{1}{y_j} \hat{N}^{+}_j, \sum_{j=1}^r y_j\hat{N}^-_j]=H_{(r)},
\quad [H_{(r)},\sum_{j=1}^r\hat{N}^{\pm}_j]=\pm 2\sum_{j=1}^r\hat{N}^{\pm}_{j}.$$
and hence 
$(\sum_{j=1}^r y_j\hat{N}^-_j, H_{(r)}, \sum_{j=1}^r \frac{1}{y_j}\hat{N}^+_{j})$ 
is also an $sl_2$-triple.

Let $g_j(y_1/y_{j+1},\ldots ,y_j/y_{j+1})$ for $1\leq j\leq r-1$ be the 
$G_{\RR}$-valued functions defined in the $SL_2$-orbit theorem \cite{CKS}.
We define
$$g(y)=\prod_{j=1}^{r-1}g_j(y_1/y_{j+1},\ldots ,y_j/y_{j+1})$$
where $y_1/y_2,\ldots ,y_r/y_{r+1}$ are sufficiently large.
We write $N(y)=\sum_j y_jN_j$.
By the $SL_2$-orbit theorem, $g(y)$ commutes with $H_{(r)}$ and 
$N(y)=\Ad{(g(y))}\hat{N}^-(y)$ where $\hat{N}^-(y)=\sum_{j=1}^{r}y_j\hat{N}^-_j$.
Therefore,
\begin{align}\label{triple}
\Ad{(g(y))}(\sum_{j=1}^r y_j\hat{N}^-_j, H_{(r)},
            \sum_{j=1}^r \frac{1}{y_j}\hat{N}^+_{j}) =(N(y),H_{(r)}, N^{+}(y))
\end{align}
is an $sl_2$-triple where 
$N^{+}(y)=\Ad{(g(y))}\sum_{j=1}^r \frac{1}{y_j}\hat{N}^+_{j}$.

\par Now $(N(y),F)$ generates a one variable nilpotent orbit with associated
$SL_2$-orbit defined by $(N(y),\hat F)$ where 
$$
     (\hat F,W(N)[-w]) = (e^{-\epsilon}.F,W(N)[-w])
$$
is the $sl_2$-splitting of $(F,W(N)[-w])$, for $N$ any element in the 
monodromy cone of positive linear combinations of $N_1,\dots,N_r$. 
The corresponding semisimple element $H$ of this $SL_2$-orbit coincides 
with $H_{(r)}$ above, and hence the associated $sl_2$-triple is 
$(N(y),H_{(r)},N^+(y))$.

\par As in the proof of Lemma $(3.12)$~\cite[p. 478]{CKS}, it then follows
that
\begin{align*}
\exp{(zN(y))}\hat{F} =\exp{(\frac{1}{z} N^+(y))}\Phi,\qquad
\Phi^p = \sum_{b\leq w-p}I^{a,b}_{(\hat{F},W(N)[-w])}
\end{align*}
Substituting $z=i$, we have
\begin{align}\label{nilp-conv}
\exp{(iN(y))}\hat{F} =\exp{(-i N^+(y))}.\Phi
\end{align}
Therefore, along any sequence $y(m)$ to which the $SL_2$-orbit theorem applies 
(i.e. $y_j(m)/y_{j+1}(m)$ sufficiently large as $m\to\infty$),
$$
    e^{iN(y(m))}.F = e^{\epsilon}e^{iN(y(m))}.\hat F = e^{\epsilon}e^{-iN^+(y(m))}.\Phi
$$
and hence the limit as $m\to\infty$ is again $\Phi$.

\subsection{Satake Boundary Components}\label{Satake}  As noted in the introduction, by the
work of Kato and Usui, the theory of degenerations of Hodge structure can
be used to construct generalizations of the toroidal compactifications 
of~\cite{AMRT} by adjoining spaces of nilpotent orbits as boundary 
components.  In analogy with this classical theory, it is natural to consider 
the reduced limit period map as providing a generalization of the Satake
construction (see \cite[Theorem 5.21(b)]{KP2} for detail).  

\par To obtain an explicit connection with the standard Satake construction,
let $\mathcal{D}$ be a classifying space of effective Hodge structures of
weight $2k-1$ upon which the Lie group $G_{\mathbb R}$ acts transitively.  
Let $V\subset G_{\mathbb R}$ be the stabilizer of a reference Hodge structure
$F\in\mathcal{D}$ and $K$ be the stabilizer of 
$$
       S = H^{2k-1,0}\oplus H^{2k-3,2}\oplus\cdots
$$
Then, $K$ is a maximal compact subgroup of $G_{\mathbb R}$ containing $V$
and the quotient $G_{\mathbb R}/K$ is a Siegel space $\mathcal{H}$ 
parametrizing the Hodge structures $S\oplus\bar S$ polarized 
by $Q$.  Via a Tate-twist, we  now renormalize $\mathcal{D}$ and 
$\mathcal{H}$ to have weight $-1$ and let $p:\mathcal{D}\to\mathcal{H}$ 
denote the real-analytic quotient map $G_{\mathbb R}/V\to G_{\mathbb R}/K$, i.e.
$$
    F^0 p\left(\{H^{\bullet,\bullet}\}\right) 
        = \bigoplus_{\text{$p$ even}}\, H^{p,-p-1}.
$$

\par Let $(N_1,\ldots ,N_r,F)$ be a nilpotent orbit with nilpotent 
cone $\sigma=\sum_{j=1}^{r} \RR_{\geq 0}N_j$ of even-type as defined in 
\cite[Definition~2.6]{H1} or \cite[Lemma~2.4]{Ha2}. 
That is $N^2=0$ for $N\in\sigma$ and $I^{p,-1-p}_{(F,W)}=0$ if $p$ is odd where $W=W(\sigma)[1]$ is the 
monodromy weight filtration.
Then we have the Satake 
boundary component  
$$
  B_S(\sigma)=\{F\in \mathit{cl}( \mathcal{H}) \;|
                \; F^{0}\cap \overline{F^{0}}=W_{-2,\CC}\}
$$
contained in the closure $\mathit{cl}( \mathcal{H})$ of the compact dual 
$\check{\mathcal{H}}$ of $\mathcal{H}$.

\par On the other hand, we have the Kato-Usui boundary component $B(\sigma)$ 
of $\mathcal{D}$ which is the set of $\sigma$-nilpotent orbits modulo 
$\exp{(\sigma_{\CC})}$. In \cite[Corollary~2.8]{H1}, the first author defined 
a map  $p_{\sigma}:B(\sigma)\to B_S(\sigma)$, where 
$\Psi:=p_{\sigma}(\sigma,F)$ is given by
\begin{equation}
   \Psi^0=(\bigoplus_{p \text{ even}}I^{p,-1-p}_{(F,W)})\oplus W_{-2,\CC}
    \label{eq:reduced-Psi}
\end{equation}
{\it n.b.} $p_{\sigma}$ is well defined since 
$e^{N}\Psi=\Psi$ for all $N\in\sigma_{\CC}$.

\par Let $(\sigma,\hat F)$ generate a nilpotent orbit of even type with
limit mixed Hodge structure split over $\mathbb R$ and 
\begin{equation}
     \tilde F^0 = (\bigoplus_{\text{$p$ even}}\, I^{p,-p-1}_{(\hat{F},W)}) 
                  \oplus(\bigoplus_p\, I^{p,-p}_{(\hat{F},W)})        \label{eq:reduced-F}
\end{equation}
Then, by Proposition $(2.10)$ of~\cite{H1} it follows that
$(\sigma,\tilde F)$ generates a nilpotent orbit 
$e^{\sum_j\, z_j N_j}.\tilde F\in\mathcal{H}$ with limit split over $\mathbb R$ such that
$$
      p(e^{\sum_j\, z_j N_j}\hat{F})= e^{\sum_j\, z_j N_j}\tilde F
$$
By the results of \S 6.1, the orbit $ e^{\sum_j\, z_j N_j}\tilde F$ has
a reduced limit $\tilde \Phi$.  Comparing the defining equations
\eqref{eq:naive-1}, \eqref{eq:reduced-Psi} and \eqref{eq:reduced-F}
it follows that 
$$
               \tilde\Phi = \Psi\in B_S(\sigma)
$$
and hence $p_{\sigma}(\sigma,\hat F)$ is equal to the reduced limit of 
the orbit $(\sigma,\tilde F)$.  The same conclusion holds for 
odd type $\sigma$ defined in~\cite{H1}.

\subsection{Complement to~\cite{CK}}  In this subsection, we prove the
following consequence of equation~\eqref{der-Gamma}

\begin{thm}\label{thm:lim2.5} For any choice of branch cuts for 
$z_j = \frac{1}{2\pi i}\log(s_j)$, 
$$
       \lim_{s\to 0}\, \Ad(e^{\sum_j z_j N_j})\Gamma(s) = 0                
$$
\end{thm}

Define the length of $\mathfrak q$ to be the smallest non-negative integer
$\lambda$ such that $\wp_{-a} = 0$ for $a>\lambda$.  Recall that
\begin{itemize}
\item[(i)] $\mathfrak q = \oplus_{a<0}\,\wp_a$;
\item[(ii)] $[\wp_a,\wp_b] \subseteq \wp_{a+b}$;
\item[(iii)] $N_j\in\wp_{-1}$;
\item[(iv)] $[N_j,\Gamma|_{s_j=0}] = 0$.
\end{itemize}

Given a multi-index $J=(a_1,\dots,a_r)$ with non-negative
entries define 
$$
    \deg_j(J) = a_j,\qquad (\log s)^J = \Pi_j\, (\log s_j)^{\deg_j(J)},\qquad
    \deg(J) = \sum_j\, \deg_j(J)
$$ 
By the commutativity of the 
$N_j$'s, and properties (i)--(iii) above, we have
\begin{equation}
         \Ad(e^{\sum_j z_j N_j})\Gamma(s)
         = \left(\sum_{\deg(J)<\lambda}\, (\log s)^J A_J\right) \Gamma(s) 
         \label{eq:lim2.6}
\end{equation}
where $A_J = b_J \Pi\, (\ad N_j)^{\deg_j(J)}$ for an appropriate constant
$b_J>0$ and $\lambda$ is the length of $\mathfrak q$.

\par Given a $\mathfrak{q}$-valued holomorphic function $f(s_1,\dots,s_r)$ which vanishes at
$s=0$ let 
\begin{equation}
       f(s) =\sum_{\deg(K)>0}\, c_K s^K f_K             \label{eq:lim2.7}
\end{equation}
denote the Taylor expansion of $f$ about $s=0$ where 
$s^K = \Pi\, s_j^{\deg_j(K)}$ and 
$$
      f_K = \left.
            \frac{\pd^{\deg(K)}f}{\pd s_1^{\deg_1(K)}\cdots\pd s_r^{\deg_r(K)}}
            \right|_0
$$
Combining \eqref{eq:lim2.6} and \eqref{eq:lim2.7} and using the fact that 
$
        \lim_{t\to 0} |t|^{\ell}\log|t| = 0
$ 
for $\ell>0$, it then follows that in order to show that \eqref{eq:lim2.6}
converges to zero, it is sufficient to prove [see Lemma\eqref{lem:lim2.9}] 
that 
\begin{equation}
      \deg_j(K) = 0 \implies [N_j,\Gamma_K] = 0            \label{eq:lim2.8}
\end{equation}
To verify $\eqref{eq:lim2.8}$, we note that $\deg_j(K)=0$ implies that
$$
      \Gamma_K = \left.\frac{\pd^K}{\pd s^K}\Gamma |_{s_j=0}\right|_0        
$$
Therefore, by $(iv)$ and the fact that Lie brackets commute with
derivatives, we have $[N_j,\Gamma_K] = 0$. To complete the proof of Theorem 
\eqref{thm:lim2.5} it remains to establish:

\begin{lem}\label{lem:lim2.9} $(\log s)^J A_J \Gamma(s)\to 0$ for any
multi-index $J$ with non-negative entries.
\end{lem}
\begin{proof} We can assume that $\deg(J)<\lambda$.  Otherwise 
$A_J\Gamma(s) = 0$. For a pair of multi-indices $J$ and $K$ define
$$
       J\cdot K = \max(\deg_j(J) \mid \deg_j(K) = 0)
$$
Then, $J\cdot K >0$ if and only if there is some index $j$ such that
$\deg_j(J)>0$ and $\deg_j(K) = 0$.  Therefore, by \eqref{eq:lim2.8} and the
commutativity of $N_i$'s:
$$
      J\cdot K> 0 \implies (A_J\Gamma)_K 
                            = b_J \Pi\, (\ad N_i)^{\deg_i(J)}\Gamma_K = 0
$$
As such,
\begin{equation}
     A_J\Gamma(s) 
        = \sum_{\deg(K)>0,\, J\cdot K=0}\,
             c_K s^K (A_J\Gamma)_K                      \label{eq:lim2.10}
\end{equation}

\par To continue, let $|J| = \{j \mid \deg_j(J)>0\}$ and 
$s^{|J|}=\Pi_{j\in |J|}\, s_j$.  Observe that 
$J\cdot K = 0$ implies that $s^{|J|} | s^K$ in $\CC[s_1,\dots,s_r]$ and hence
by \eqref{eq:lim2.10},
\begin{equation}
        s^{|J|} | A_J\Gamma(s)                         \label{eq:lim2.11}
\end{equation}
in $\mathcal O$.  Therefore,
$$
      (\log s)^J A_J \Gamma(s) 
        = ((\log s)^J s^{|J|})(A_J\Gamma(s)/s^{|J|})
$$
where the first factor $((\log s)^J s^{|J|})\to 0$ as $s\to 0$ by 
elementary calculus and the second factor $A_J\Gamma(s)/s^{|J|}$ is
bounded near $s=0$ by \eqref{eq:lim2.11}.
\end{proof}

\subsection{Convergence of Reduced Limit} Let $F:U^r\to\mathcal D$ be the period map of a variation of pure
Hodge structure of weight $k$ over $\Delta^{*r}$ as in \eqref{eq:basic-diagram} with associated nilpotent 
orbit generated by $(N_1,\dots,N_r;F)$.  Let $\Phi$ be the naive limit filtration \eqref{eq:naive-2} attached 
to the limit mixed Hodge structure 
$$
          (F,W(N)[-k]),\qquad N = \sum_j\, N_j
$$
In this section we prove: 

\begin{thm}\label{thm:rl-0} If $z(m) = x(m) + iy(m)\in U^r$ is a sequence of points such that
\begin{itemize}
\item[(a)] $x(m)$ is bounded;
\item[(b)] Each component $y_j(m)$ of $y(m)$ diverges to $\infty$;
\end{itemize} 
Then $F(z(m))\to\Phi$.
\end{thm}

\par The proof is by a series of Lemmata.  To state the first lemma, for any vector
$v = (v_1,\dots,v_r)\in\mathbb C^r$ define
$$
       N(v) = \sum_j\, v_j N_j
$$
As in the previous sections, $F(z) = e^{N(z)}e^{\Gamma(s)}.F$ will denote the local normal form of $F(z)$.
Accordingly, given $z(m)$ we define $s(m)$ to be the corresponding sequence in $\Delta^{*r}$ obtained by setting
$s_j(m) = e^{2\pi iz_j(m)}$.

\begin{lem}\label{lem:rl-A} Let $z(m)$ be a sequence as above and suppose that $e^{N(z(m))}F \to \Phi$.
Then,
\begin{equation}
         F(z(m))\to \Phi         \label{eq:rl-1}
\end{equation}
\end{lem}
\begin{proof} Using the local normal form, we have
$$
       F(z) = \left(e^{\sum_j\, z_j N_j}e^{\Gamma(s)}e^{-\sum_j\, z_j N_j}\right)
               e^{\sum_j\, z_j N_j}.F
$$
By \S 6.3, 
$$
      \Ad(e^{N(z(m))})e^{\Gamma(s(m))} \to 1
$$
whereas $e^{N(z(m))}.F\to\Phi$ by assumption.  Combining these two observations gives \eqref{eq:rl-1}.
\end{proof}

\par The next step is to show that $e^{N(z(m))}.F\to\Phi$.  This was shown in \S 6.1 provided that 
$y_j(m)/y_{j+1}(m)\to\infty$ for $j=1,\dots,r$, where by convention we set $y_{r+1}(m) = 1$.
More generally, there is a constant $K$ such that the results of \S 6.1 hold whenever $y_j(m)/y_{j+1}(m)>K$.

\par To handle the general case, we import the notion of an $sl_2$-sequence from \cite{BP1}:  First, we can
partition $z(m)$ into a collection of subsequences on which some permutation of the following system of 
inequalities holds:
\begin{equation}
        y_1(m)\geq y_2(m)\geq\cdots\geq y_r(m)                   \label{eq:rl-2}
\end{equation}
Accordingly, we assume \eqref{eq:rl-2} hold for the remainder of this section.  Second, we can \lq\lq group\rq\rq{} strings of $y_j(m)$'s together which go to infinity as the same rate.  More precisely, we say that $y(m)$ is an 
$sl_2$-sequence if there exists
\begin{itemize}
\item[(i)]   A linear transformation $T:\mathbb R^d\to\mathbb R^r$;
\item[(ii)]  A sequence $v(m)\in\mathbb R^d_{\geq 0}$;
\item[(iii)] A convergent sequence $b(m)\in\mathbb R^r_{\geq 0}$
such that
$$
       y(m) = T(v(m)) + b(m)
$$
and $v_j(m)/v_{j+1}(m)\to\infty$.
\end{itemize}

\begin{lem} Let $y(m)$ be an $sl_2$-sequence such that $y_r(m)\to\infty$. Then,
$$
       e^{N(iy(m))}.F\to \Phi
$$
\end{lem} 
\begin{proof} Let $T:\mathbb R^d\to\mathbb R^r$ be the linear transformation associated to $y(m)$, 
$e_j\in\mathbb R^d$ be the j'th unit vector and $\theta^j = T(e_j)$.  Then,
$$
     N(y(m)) = \sum_j\, v_j(m)N(\theta^j) + N(b(m))
$$
The $SL_2$-orbit theorem now applies to the nilpotent orbit generated by the $N(\theta^j)$'s and $F$.  Therefore,
$$
      e^{iN(y(m))}F = e^{iN(b(m))}e^{\sum_j\, v_j(m)N(\theta^j)}F
      \to\Phi
$$
using the results of \S 6.1, because
\begin{itemize}
\item[(1)]  $\sum_j\,N(\theta^j)$ belongs to the interior of the monodromy cone of $N_1,\dots,N_r$ and hence
$$
          W(\sum_j\,N(\theta^j)) = W(N_1 + \cdots N_r)
$$
\item[(2)]  $\Phi$ depends only on the limit mixed Hodge structure $(F,W)$ of the underlying nilpotent orbit;
\item[(3)]  $N(b(m))$ converges to an element of type $(-1,-1)$ and all such elements preserve $\Phi$.
\end{itemize}     
\end{proof}

\par To finish the proof of Theorem \eqref{thm:rl-0}, suppose that $y(m)$ is a sequence of points satisfying
the system of inequalities \eqref{eq:rl-2} with $y_r(m)\to\infty$ for which $e^{iN(y(m))}.F$ does not converge
to $\Phi$.  Then, we can find a subsequence $y'(m)$ such that $e^{iN(y'(m))}.F$ remains outside a fixed open
neighborhood of $\Phi$.  Induction on the number of variables shows that we can always find a subsequence
$y''(m)$ of $y'(m)$ which is an $sl_2$-sequence.  By the above, $e^{iN(y''(m))}.F\to\Phi$, which contradicts
the fact that $e^{iN(y'(m))}.F$ avoids a neighborhood of $\Phi$.  Thus, $e^{iN(y(m))}.F\to\Phi$ and hence 
Lemma \eqref{lem:rl-A} implies Theorem \eqref{thm:rl-0}.

\subsection{Variations of Mixed Hodge Structure} In this section, we show that if $\theta(z) = (e^{zN}.F,W)$ is an 
admissible nilpotent orbit of graded-polarized mixed Hodge structure, and $\hat\theta(z)$ is the map obtained by
applying \eqref{eq:canonical-splitting} to $\theta(z)$ then
$$
              \Phi = \lim_{y\to\infty}\, \hat\theta(iy)
$$
exists. More generally, we prove an analogous statement for the map $F:U\to\mathcal M$ attached 
\eqref{eq:mixed-diagram} to an admissible variation of mixed Hodge structure $\mathcal V\to\Delta^*$ with 
unipotent monodromy.

\par In Example \eqref{exm:non-inv} we show that $\Phi$ need not be invariant under $N$.  In Example 
\eqref{exm:non-conv} we show that in several variables the reduced limit depends on how the underlying 
sequence of points in $U^r$ goes to infinity.  

\par Unless otherwise noted, for the remainder of this section $z = x+iy$ is a point in the upper half-plane $U$, 
and $z(m) = x(m) + iy(m)$ is a sequence of points in $U$ such that $x(m)$ is convergent and $y(m)\to\infty$.
Let $(\hat F(z),W)$ be the associated family of mixed Hodge structures obtained by applying 
\eqref{eq:canonical-splitting} to $(F(z),W)$.  Define 
$$
           \tilde F(z) = e^{-xN}\hat F(z),\qquad 
           \tilde Y(z) = Y_{(\tilde F(z),W)} = \Ad(e^{-xN})Y_{(\hat F(z),W)}
$$ 
where $N$ is the monodromy logarithm of $\mathcal V$.  In the case where $F(z)$ is the nilpotent orbit 
$\theta(z)$, $\tilde F(z) = \hat\theta(iy)$.  Theorem $(2.30)$ of \cite{BP1} implies: 

\begin{lem} 
\begin{equation}
           \lim_{m\to\infty}\, \tilde Y(z(m)) = Y^0           \label{eq:thm2.30}
\end{equation}
where $Y^0$ is the grading of $W=W^0$ appearing in \eqref{eq:t-y}.
\end{lem}

Likewise, we have:

\begin{lem} 
\begin{equation}
            \lim_{m\to\infty}\, \tilde F(z) Gr^W_k = \lim_{m\to\infty} F(z(m))Gr^W \label{eq:same-lim}
\end{equation}
\end{lem}
\begin{proof} Since $\hat F(z)$ and $F(z)$ induce the same filtration on $Gr^W$, the assertion reduces to 
the statement that for a variation of pure Hodge structure
$$
         \lim_{m\to\infty}\, e^{-x(m)N}F(z(m))
$$
converges.  But, equation \eqref{eq:naive-2} implies that $\Phi = \lim_{m\to\infty}\, F(z(m))$ is fixed by $N$. 
Therefore, since $x(m)$ is convergent, $e^{-x(m)N}F(z(m))\to\Phi$ as well.
\end{proof}

\par In particular, since both $\tilde Y(z(m))$ and $\tilde F(z(m))Gr^W$ converge, it follows that $\tilde F(z(m))$. 
converges.  To identify this limiting filtration, let $W^1 = M(N,W^0)$ and $(F_{\infty},W^1)$ denote the limit mixed Hodge structure of $\mathcal V$.  Let 
\begin{equation}
            (\hat F_{\infty},W^1) = (e^{-\epsilon}.F_{\infty},W^1)						\label{eq:split-lmhs}
\end{equation}
be the splitting \eqref{eq:canonical-splitting} of $(F_{\infty},W^1)$.  

\begin{lem} $Y^0$ preserves the Deligne bigrading of  $(\hat F_{\infty},W^1)$.  The
splitting \eqref{eq:canonical-splitting} $(F_{\infty}Gr^W_k,W^1 Gr^W_k)$ is $(\hat F_{\infty}Gr^W_k, W^1 Gr^W_k)$ 
\end{lem}
\begin{proof} This is a special case of Lemma $(5.7)$ and Corollary $(5.9)$ of \cite{BP1}.
\end{proof}

\par By \eqref{eq:same-lim}, $\tilde F(z(m))Gr^W_k \to \Phi_k$ where $\Phi_k$ is the filtration defined by 
\eqref{eq:naive-2} and the splitting of the limit mixed Hodge structure of $\mathcal V$ on $Gr^W_k$.  By the
previous lemma, this splitting is $(\hat F_{\infty} Gr^W_k, W^1 Gr^W_k)$. Combining the above, we obtain:

\begin{thm} Let $x(m)\to x(\infty)$.  Then,
\begin{equation}
          \tilde F^p(z(m)) \to \Phi :=\bigoplus_{k,s\leq k-p}\, I^{r,s}_{(\hat F_{\infty},W^1)}\cap E_k(Y^0)
                                                                                       \label{eq:avmhs-rlim}
\end{equation}
where $E_k(Y^0)$ is the $k$-eigenspace of $Y^0$, and hence
\begin{equation}
        \hat F(z(m)) = e^{x(m)N}\tilde F(z(m))\to e^{x(\infty)N}.\Phi                   \label{eq:avmhs-rlim-2}
\end{equation}
\end{thm}

\par In the pure case, $N$ preserves $\Phi$. The following example shows that this is not true in the mixed case,
and hence \eqref{eq:avmhs-rlim-2} can depend on $x(\infty)$.

\begin{exm}\label{exm:non-inv} Let $\theta(z) = (e^{zN}.F,W)$ be an admissible nilpotent orbit with limit mixed 
Hodge structure split over $\mathbb R$.  Then, by a result of Deligne (cf. \cite{BP1}), 
$$
           \hat\theta(iy) = e^{iyN_0}.F
$$
is the splitting \eqref{eq:canonical-splitting} where 
$$
           N = N_0 + N_{-2} + \cdots,\qquad [Y^0,N_{-j}] = -j N_{-j}
$$
Let $e^{zN}.F$ be the admissible nilpotent orbit with underlying real vector space $V_{\mathbb R}$
spanned by $e_0,\dots,e_3$ and limit mixed Hodge structure $(F,W)$ determined by
$$
     I^{0,0} = \mathbb C e_0\oplus\mathbb C e_1,\qquad
     I^{-1,-1} = \mathbb C e_2,\qquad I^{-2,-2} = \mathbb C e_3
$$
with weight filtration $W_0 = V_{\mathbb C}$, $W_{-1} = W_{-2} = \bigoplus_{j>0}\, \mathbb C e_j$, $W_{-3} = 0$,
and monodromy logarithm
$$
     N(e_0) = e_2,\qquad N(e_1) = e_2,\qquad N(e_2) = e_3, \qquad N(e_3) = 0
$$
Then, $E_0(Y^0) = \mathbb C e_0$, $E_{-2}(Y^0) = W_{-2}$.  By \eqref{eq:avmhs-rlim}
$$
       \Phi^0 = (\bigoplus_{s\leq 0}\, I^{r,s}\cap E_0(Y^0))\oplus (\bigoplus_{s\leq -2}\, I^{r,s}\cap E_{-2}(Y^0))
              = \mathbb C e_0\oplus\mathbb C e_3
$$
which is not invariant under $N$.      
\end{exm}

\begin{rem} Example \eqref{exm:non-inv} also shows that 
$$
     \lim_{y\to\infty}\, e^{iyN}F^0 = \mathbb C e_2\oplus\mathbb C e_3
$$
which induces the trivial filtration on $Gr^W_0$ and hence does not belong to $\check{\mathcal M}$.
\end{rem}


\begin{thm} Let $\mathcal V\to\Delta^*$ be an admissible normal function with unipotent monodromy.  Then, 
$\hat F(z) = F(z)$ due to the short length of $W$, and $F(z(m))\to\Phi$ for any sequence $z(m)$ such that
$y(m)\to\infty$ and $x(m)$ is bounded.
\end{thm}
\begin{proof} The fact that $Y_{(F(z),W)}$ has a well defined limit which commutes with $N$ in the case where
$(F(z),W)$ arises from an admissible normal function is Theorem $(3.9)$ of \cite{BP}.  Since $F(z)Gr^W$ exists,
this shows that $\lim_{\text{Im}(z)\to\infty}\, F(z)$ exists.
\end{proof}

\par In order to obtain reduced limits for admissible normal functions in several variables, it would be
sufficient to know that $Y_{(F(z_1,\dots,z_r),W)}$ converged to a grading of $W$ which belongs to the kernel
of each $\ad N_j$.  If $(N_1,\dots,N_r)$ are the monodromy logarithms of a normal function which is singular
in the sense of Green and Griffiths, there is no grading of $W$ which belongs to the kernel of each $\ad N_j$.
We conclude with a simple example of a normal function which is singular in this sense, and show that the 
limit depends on the path to infinity:

\begin{exm}\label{exm:non-conv} Let $F(z_1,z_2)$ be the nilpotent orbit with underlying real vector space 
$V_{\mathbb R}$ spanned by $e_0,e_1,e_2$, limit mixed Hodge structure determined by the bigrading 
$$
   I^{0,0} = \mathbb C e_1\oplus\mathbb C e_2,\qquad I^{-1,-1} = \mathbb C e_2
$$
with weight filtration $W_0 = V_{\mathbb C}$ and $W_{-1} = \mathbb C e_1\oplus\mathbb C e_2$, and monodromy 
logarithms
$$
\aligned
      N_1(e_0) &= e_2,\qquad N_1(e_1) = e_2,\qquad N_1(e_2) = 0 \\
      N_2(e_0) &= -e_2,\qquad N_2(e_1) = e_2,\qquad N_2(e_2) = 0 
\endaligned
$$
Then, 
$$
      e^{iy(N_1 + N_2)}.F_{\infty}^0 = \mathbb C e_0 \oplus\mathbb C(e_1 + 2iy e_2)
$$
which converges to $\mathbb C e_0\oplus\mathbb C e_2$ as $y\to\infty$.  On the other hand,
$$
      e^{iy^2 N_1 + iy N_2}.F_{\infty} = \mathbb C (e_0 + iy(y-1)e_2)\oplus\mathbb C(e_1 + iy(y+1)e_2)
$$
To calculate the limit as $y\to\infty$, observe that
$$
      (y+1)(e_0 + iy(y-1) e_2) - (y-1)(e_1 + iy(y+1)e_2) = (y+1)e_0 - (y-1)e_1
$$
and hence 
$$
       e^{iy^2 N_1 + iy N_2}.F_{\infty} \to \mathbb C (e_0 -e_1)\oplus\mathbb C e_2
$$
\end{exm}

\end{document}